\documentclass[11pt]{article}

\usepackage{color,mathrsfs,amsthm}

\usepackage{amsmath}

\expandafter\let\csname equation*\endcsname\relax

\expandafter\let\csname endequation*\endcsname\relax

\usepackage{amssymb}

\usepackage{latexsym}

\usepackage{epsfig}

\usepackage{graphicx}

\usepackage{psfrag}

\newcommand{\R}{{\mathbb R}}

\newcommand{\D}{{\cal D}}

\newcommand{\Id}{\mbox{\it Id}}

\renewcommand{\d}{\partial}

\newcommand{\eps}{\varepsilon}

\def\hat{\widehat}

\def\tilde{\widetilde}

\def\bfo{\begin {eqnarray*} }

\def\efo{\end {eqnarray*} }

\def\ba{\begin {eqnarray*} }

\def\ea{\end {eqnarray*} }

\def\beq{\begin {eqnarray}}

\def\eeq{\end {eqnarray}}

\def\supp{\hbox{supp}\,}

\def\dist{\hbox{dist}}

\newcommand{\pa}{\partial}

\def\RR{{\mathbb R}}

\def\RR{{\mathbb R}}

\renewcommand{\r}[1]{(\ref{#1})}

\newcommand{\be}[1]{\begin{equation}\label{#1}}

\newcommand{\ee}{\end{equation}}

\renewcommand{\d}{\mathrm{d}}

\newcommand{\bo}{\partial M}

\newcommand{\Mint}{M^\text{\rm int}}

\newcommand{\CI}{C^\infty}

\newcommand{\ep}{\epsilon}

\newcommand{\scl}{{\mathrm{sc}}}

\newcommand{\Psisc}{\Psi_\scl}

\newcommand{\cM}{\mathcal M}

\newcommand{\loc}{{\mathrm{loc}}}

\newtheorem{lemma}{Lemma}[section]

\newtheorem{prop}[lemma]{Proposition}

\newtheorem{thm}[lemma]{Theorem}

\newtheorem{cor}[lemma]{Corollary}

\newtheorem*{thm*}{Theorem}

\newtheorem*{prop*}{Proposition}

\newtheorem*{cor*}{Corollary}

\newtheorem*{conj*}{Conjecture}

\numberwithin{equation}{section}

\theoremstyle{remark}

\newtheorem{rem}[lemma]{Remark}

\newtheorem*{rem*}{Remark}

\theoremstyle{definition}

\newtheorem{Def}[lemma]{Definition}

\newtheorem*{Def*}{Definition}

\newcommand{\Tsc}{{}^{\scl}T}
\newcommand{\Sym}{\mathrm{Sym}}

\renewcommand{\r}[1]{(\ref{#1})} 
\newcommand\Cx{\mathbb{C}}
\newcommand\foliation{\mathsf{x}}

\newcommand\level{\mathsf{c}}
\newcommand\Omegaext{\hat\Omega}

\renewcommand{\d}{\mathrm{d}}

 \newcommand{\zero}{^{(0)}}

\renewcommand{\phi}{\varphi}

\title{Journey to the Center of the Earth}

\author{Gunther Uhlmann\footnote{Department of Mathematics, University of Washington, Seattle, WA 98195-4350, USA; Department of Mathematics, University of Helsinki, Helsinki, Finland; HKUST Jockey Club Institute for Advanced Study, HKUST, Clear Water Bay, Kowloon, Hong Kong, China. {\em Email:} gunther@math.washington.edu} and Hanming Zhou\footnote{Department of Pure Mathematics and Mathematical Statistics, University of Cambridge, Cambridge CB3 0WB, United Kingdom. {\em Email:} hz318@dpmms.cam.ac.uk}}

\date{\vspace{-1cm}}

\begin{document}

\maketitle

\begin{abstract}
We survey some results on travel time tomography. The question is whether
we can  determine the anisotropic index of refraction of a medium by measuring the travel times of waves going through the medium. This can be recast as  geometry problems, the boundary rigidity problem and the lens rigidity problem. The
boundary rigidity problem is whether we can  determine a Riemannian metric of
a compact Riemannian manifold with boundary by measuring the distance function between boundary points. The lens rigidity problem problem is to determine a Riemannian
metric of a Riemannian manifold with boundary by measuring for every point and direction of entrance of a geodesic the point of exit and direction of exit and its length. The linearization of these two problems is tensor tomography. The question is whether one can determine a symmetric two-tensor from its integrals along geodesics. We emphasize recent results on boundary and lens rigidity
and in tensor tomography in the partial data case.
\end{abstract}

\section{Introduction}

The question of determining the sound speed or index of refraction
of a medium by measuring the first arrival times of waves arose in
geophysics in an attempt to determine the substructure of the
Earth by measuring at the surface of the Earth the travel times of
seismic waves. An early success of this inverse method was the
estimate by Herglotz \cite{Her} and Wiechert and Zoeppritz \cite{WZ} of the
diameter of the Earth and the location of the mantle, crust and
core. The assumption used in those papers is that the index of
refraction (which is inverse proportional to the speed) depends only on the radius. A more
realistic model is to assume that it depends on position, the case of an heterogeneous medium. The
travel time tomography problem can be formulated mathematically as
determining a Riemannian metric on a bounded domain (the Earth)
given by $ds^2= \frac{1}{c^2(x)}dx^2$, where $c$ is a positive
function, from the length of geodesics (travel times) joining
points in the boundary.

More recently it has been realized,  by measuring the travel times
of seismic waves, that the inner core of the Earth exhibits
anisotropic behavior, that is the speed of waves depends also on
direction there with the fast direction parallel to the Earth's
spin axis \cite{Cre}. 
Given the complications presented by
modeling the Earth as an anisotropic elastic medium we consider a
simpler model of anisotropy, namely that the wave speed is given
by a symmetric, positive definite matrix $g=(g_{ij})(x),$ that is,
a Riemannian metric in mathematical terms. The problem is to
determine the metric from the lengths of geodesics joining points
in the boundary (the surface of the Earth in the motivating
example). Other applications of travel time tomography are to imaging the Sun's interior \cite{Ko}, medical imaging \cite{SW}  and to ocean acoustics \cite{MW}
to
name a few. 

A general and geometric formulation of the travel time tomography problem is 
the question of whether given a compact Riemannian
manifold with boundary one can determine the Riemannian metric
in the interior knowing the lengths of geodesics joining points on the
boundary, i.e. the boundary distance function. This is a problem
that also appears naturally in rigidity questions in Riemannian geometry and it is known as the 
{\em boundary rigidity problem}. 
Notice that the boundary distance function is unchanged under any isometry which fixes the
boundary, thus one can only expect to recover the metric up to this natural obstruction.

The boundary distance function takes into account only length minimizing geodesics, one
can consider the behavior of all geodesics going through the manifold. This induces another type of rigidity problems: the {\em lens rigidity problem} and {\em scattering rigidity problem}, which concerns the determination of a Riemannian metric up to the natural obstruction, from the scattering relation or lens data. The scattering relation, introduced by Guillemin \cite{G}, is a map which sends the point and direction of entrance of a geodesic to point and direction of exit. The scattering relation together with information of lengths of geodesics gives the lens data. Again, lens data is unchanged under an isometry fixing the boundary.

There is another closely related problem, the {\em geodesic X-ray transform}, where
one integrates a function or a tensor field along geodesics of a Riemannian metric. The integration of a function along geodesics is
the linearization of the boundary rigidity problem and lens rigidity problem in a fixed conformal class. The standard
X-ray transform \cite{Hel}, where one integrates a function along straight lines, corresponds to
the case of the Euclidean metric and is the basis of medical imaging techniques such as
CT and PET. The case of integration along a general geodesic arises in geophysical and
ultrasound imaging. The case of integrating tensors of order two along geodesics, also
known as {\em deformation boundary rigidity}, is the linearization of the general boundary rigidity problem and lens rigidity problem. One important inverse problem for the geodesic X-ray transform is whether one can recover a function or a tensor field from its integrals over geodesics, this is the {\sl tensor tomography} problem.. We review in this article some recent results on the boundary and lens rigidity problem as well as tensor tomography when one has data on part of the boundary, the partial data problem \cite{UV2}, \cite{SUV_localrigidity}, \cite{SUV}. The partial data results have led to new global results for the lens rigidity problem.  

In section 2 we review results on the boundary rigidity problem with data on the whole boundary. This is mostly based on \cite{PU} and \cite{StU4}. In section 3 we review results on tensor tomography and in section 4, lens rigidity with full data.
In section 5 we consider partial data for boundary and lens rigidity and tensor tomography, in particular we give new results on the lens rigidity problem with full data.

In this paper we only consider the case of {\sl transmission} tomography.  For the case of {\sl reflection} tomography see for instance \cite{KLU} and \cite{CQUZ}.
\medskip

\noindent{\bf Acknowledgement:}
GU was partly supported by NSF and HZ was supported by EPSRC grant EP/M023842/1.
GU takes this opportunity to thank his collaborators M. Lassas, G. Paternain, L. Pestov, M. Salo, V. Sharafutdinov, P. Stefanov and A. Vasy who have enriched tremendously his understanding of the subject of this paper. HZ thanks G. Paternain and M. Salo for many helpful discussions on related topics.

%%%%%%%%%%%%%%%%%%%%%%%%%%%%%%%%%%%%%%%%%%%%%%%%%%%%%%%%%%%%%%%%%%%%%%%%%
%%%%%%%%the following: seeing the unseen %%%%%%%%%%%%%%%%%%%%%%%%%%%%%%%%%%%%%%
%%%%%%%%%%%%%%%%%%%%%%%%%%%%%%%%%%%%%%%%%%%%%%%%%%%%%%%%%%%%%%%%%%%%%%%%%%%%%%

%\subsection{Boundary  Rigidity in Two Dimensions}

%=====================================================================

%%%%%%%%%%%%%%%%%%%%%%%%%%%%%%%%%%%%%%%%%%%%%%%%%%%%%%%%%%%%%%%%%%%%%%%%%%%

\section{Boundary rigidity in the full data case}

In this section we formulate precisely the boundary rigidity problem and survey some of the main results. 

Let $(M,g)$ be a compact Riemannian manifold with boundary
$\partial M$. Let $d_g(x,y)$ denote the geodesic distance between
$x$ and $y$, two points in the boundary. This is defined as the
infimum of the length of all sufficiently smooth curves joining
the two points. The function $d_g$ measures the first arrival time
of waves joining points of the boundary. One of the inverse problems we
discuss in this section is whether we can determine the Riemannian
metric $g$ knowing $d_g(x,y)$ for any $x\in\partial M$, $y\in
\partial M$.  This problem also
arose in rigidity questions in Riemannian geometry \cite{Mi},
\cite{Cro}, \cite{Gr}.
The metric $g$ cannot be determined from this information alone.
We have $d_{\psi^*g}=d_g$ for any diffeomorphism $\psi: M\to M$
that leaves the boundary pointwise fixed, i.e., $\psi|_{\partial
M}=\Id$, where $\Id$ denotes the identity map and $\psi^*g$ is the
pull-back of the metric $g$. The natural question is whether this
is the only obstruction to unique identifiability of the metric.
It is easy to see that this is not the case.  Namely one can
construct a metric $g$ and find a point $x_0$ in $M$ so that
$d_g(x_0, \partial M)> \hbox{ sup }_{x,y \in \partial M}
d_g(x,y)$. For such a metric, $ d_g $ is independent of a change
of $ g $ in a neighborhood of $ x_0 $.  The hemisphere of the
round sphere is another example.
Therefore it is necessary to impose some a-priori restrictions on
the metric. One such restriction is to assume that the Riemannian
manifold $(M,g)$ is {\em simple}, i.e., 
any geodesic has no conjugate points and $\partial M$ is strictly
convex. $\partial M$ is strictly convex if the second fundamental
form of the boundary is positive definite at every boundary point.
R.~Michel conjectured in \cite{Mi} that simple manifolds are
boundary distance rigid, that is $d_g$ determines $g$ uniquely up
to an isometry which is the identity on the boundary. This is
known for simple subspaces of Euclidean space (see \cite{Gr}),
simple subspaces of an open hemisphere in two dimensions (see
\cite{Mi1} ), simple subspaces of symmetric spaces of constant
negative curvature \cite{BCG}, simple two dimensional spaces of
negative curvature (see \cite{Cro1} or \cite{O}). If one metric is close
to the Euclidean metric boundary rigidity was proven in \cite{LSU} that was improved in \cite{BI}. We remark that
simplicity of a compact manifold with boundary can be determined
from the boundary distance function.
%Michel's conjecture was proven in generality in \cite{PU} in two
%dimensions and we state the result in  \ref{2D boundary rigid}.
%\subsection{Boundary Rigidity in Two Dimensions}
%\begin{theorem}\label{pestovuhlmann}
%Let $(M,g_i), i=1,2$ be two dimensional simple compact Riemannian
%manifolds with boundary. Assume
%$$d_{g_1}(x,y)=d_{g_2}(x,y)
%\quad   \forall (x,y)\in\partial M\times\partial M.
%$$
%Then there exists a diffeomorphism $\psi:  M\to M$,
%$\psi|_{\partial M}=Id$, so that
%$$g_2=\psi^*g_1.$$
%\end{theorem}
In the case that both $g_1$ and $g_2$ are conformal to the
Euclidean metric $e$ (i.e., $(g_k)_{ij}= \alpha_k\delta_{ij}$,
$k=1,2$ with $\delta_{ij}$ the Kronecker symbol), as mentioned
earlier,  the problem we are considering here is known in
seismology as the inverse kinematic problem. In this case, it has
been proven by Mukhometov in two dimensions \cite{Mu} that if
$(M,g_i), i=1,2$ is simple and $d_{g_1}=d_{g_2}$, then $g_1=g_2$.
More generally the same method of proof shows that if $(M, g_i),
i=1,2,$ are simple compact Riemannian manifolds with boundary and
they are in the same conformal class then the metrics are
determined by the boundary distance function. More precisely we
have:
\begin{thm}\label{mukhometov}
Let $(M,g_i), i=1,2$ be simple compact Riemannian manifolds with
boundary of dimension $n\ge 2$. Assume $g_1=\rho g_2$ for a positive, smooth function
$\rho , \rho|_{\partial M}=1$ and $d_{g_1}=d_{g_2}$ then
$g_1=g_2$.
\end{thm}
This result and a stability estimate were proven in \cite{Mu}. We
remark that in this case the diffeomorphism $\psi$ that is present
in the general case must be the identity if the metrics are
conformal to each other. For related results and generalizations
see \cite{B}, \cite{BG}, \cite{Cro}, \cite{GN}, \cite{MuR}.

%In section \ref{2D boundary rigid} we consider the boundary rigidity in the two dimensional case and in section \ref{higher boundary rigid} in the higher dimensional case. %In section 2.4 we discuss some results on the non-simple case where the measurements are given by the {\sl scattering relation}. Roughly speaking one measures the point of exit and direction of exit of a geodesic for which we know the point of entrance and direction of entrance besides this we also know the travel time, that is the length of that geodesic.

%%%%%%%%%%%%%%%%%%%%%%%%%%%%%%%%%%%%%%%%%%%%%%%%%%%%%%%%%%%%%%%%%%%%%%%

%\subsection{The 2D result for simple manifolds}
% \label{2D boundary rigid}

In \cite{PU} it was proven Michel's conjecture in the two dimensional case:

\begin{thm}\label{pestovuhlmann}
Let $(M,g_i), i=1,2$ be two dimensional simple compact Riemannian
manifolds with boundary. Assume
$$d_{g_1}(x,y)=d_{g_2}(x,y)
\quad   \forall (x,y)\in\partial M\times\partial M.
$$
Then there exists a diffeomorphism $\psi:  M\to M$,
$\psi|_{\partial M}=Id$, so that
$$g_2=\psi^*g_1.$$
\end{thm}

The proof of Theorem \ref{pestovuhlmann} involves a connection between the
scattering relation defined in section 4 and the Dirichlet-to-Neumann map (DN)
associated to the Laplace-Beltrami operator \cite{U1}. 
define the scattering relation.
A sketch of the proof of \ref{pestovuhlmann} can be found in \cite{U1}

%\subsection{Generic local result in 3D or higher for simple manifolds}\label{higher boundary rigid}

%\subsection{Boundary Rigidity and Tensor Tomography in Dimensions $n\ge 3$}
%In \cite{StU1}, it was proven a local result for metrics in a small neighborhood%^d of the Euclidean one. This result was used in \cite{LSU} to prove a
%semiglobal solvability result assuming that one metric is
%close to the Euclidean and the other has bounded curvature.
%Burago and Ivanov have recently extended the latter result; they show that metrics close
%to the Euclidean metric are boundary rigid \cite{BI}.

\section{\bf Boundary Rigidity and Tensor Tomography}

We review here the general results obtained in \cite{StU1} for boundary rigidity and tensor tomography.

It was shown in \cite{Sh} that the linearization of the boundary rigidity problem is given by the following integral geometry problem: recover a symmetric tensor of order 2, which in any coordinate system is given by
 $f=(f_{ij})$, by the geodesic X-ray transform
$$
I_g f(\gamma) = \int f_{ij}(\gamma(t))\dot \gamma^i(t)\dot\gamma^j(t)\,\d t,$$
using the Einstein summation convention,
known for all geodesics $\gamma$ in $M$. In this section we denote by $I_g$ the geodesic X-ray transform of tensors of order two. It can be easily seen that $I_gd v=0$ for any vector field $v$ with $v|_{\partial M}=0$, where $d v$ denotes the symmetric differential
\begin{equation}   \label{dv}
[d v]_{ij} = \frac12\left(\nabla_iv_j+ \nabla_jv_i  \right),
\end{equation}
and $\nabla_k v$ denote the covariant derivatives of the vector field $v$. This is the linear version of the fact that  $d_g$ does not change on  $(\partial M)^2:=\partial M\times \partial M$ under an action of a diffeomorphism as above. The natural formulation of the linearized problem is therefore that $I_gf=0$ implies $f=dv$ with $v$ vanishing on the boundary. 
We will refer to this property as {\em s-injectivity} of $I_g$. 
More precisely, we have.
\begin{Def}
We say that $I_g$ is {\em s-injective} in $M$, if $I_{g}f=0$ and $f\in {L}^2(M)$ imply $f=d v$ with some vector field $v\in {H}_0^1(M)$.
\end{Def}
Any symmetric tensor $f\in L^2(M)$ admits an orthogonal decomposition $f = f^s+dv$ into a {\it solenoidal} and {\it potential} parts with $v\in H_0^1(M)$, and $f^s$ divergence free, i.e., $\delta f^s=0$, where $\delta$  is the adjoint operator to $-d$ given by $[\delta f]_i = g^{jk} \nabla_k f_{ij}$ \cite{Sh}. Therefore, $I_g$ is s-injective, if it is injective on the space of solenoidal tensors. 
The inversion of $I_g$ is a problem of independent interest 
in integral geometry, also called {\sl tensor tomography}.   We first survey the recent results on this problem. S-injectivity, respectively injectivity for 1-tensors (1-forms) and functions is known, see \cite{Sh} for references. 
S-injectivity of $I_g$  was proved 
in \cite{PS} for metrics with negative curvature, in \cite{Sh} for metrics with small curvature and in \cite{ShU} for Riemannian surfaces with no focal points. A conditional and non-sharp stability estimate for metrics with small curvature is also established in \cite{Sh}. In 
\cite{StU2},  stability estimates for s-injective metrics  were shown 
and sharp estimates about the recovery of a 1-form $f=f_j dx^j$ and a 
function $f$ from the associated $I_gf$ which is defined by
$$
I_g f(\gamma) = \int f_{i}(\gamma(t))\dot\gamma^i(t)\,\ dt.
$$
The stability estimates proven in \cite{StU2}  were used
to prove local uniqueness for the boundary rigidity problem near any simple metric $g$ with s-injective $I_g$. 
Similarly to \cite{Tr}, we say that $f$ is analytic in the set $K$ (not 
necessarily open), if it is real analytic in some neighborhood of $K$. 
The results that follow in this section are based on \cite{StU4}.
The first main result we discuss is about s-injectivity for  simple analytic metrics. 
\begin{thm}  \label{thm_an}
Let $g$ be a  simple,  real analytic metric in $M$. Then $I_g$ is s-injective.
\end{thm}

By proving an stability estimates in \cite{StU1} it was shown the following generic result:

\begin{thm} \label{thm_gen} There exists $k_0$ such that for each $k\ge k_0$,  the set $\mathcal{G}^k(M)$ of simple  $C^k\!(M)$  metrics in $M$ for which $I_g$ is s-injective  is  open and dense in the $C^k\! (M)$ topology.  
\end{thm}
Of course, ${\mathcal G}^k(M)$ includes all real analytic simple metrics in $M$, according to Theorem~\ref{thm_an}.

Theorem~\ref{thm_gen} %and especially estimate (\ref{basic_estimate})
allows us to  prove  the following local generic uniqueness result for the non-linear boundary rigidity problem.
\begin{thm} \label{thm_rig}
Let $k_0$ and $\mathcal{G}^k(M)$ be as in Theorem~\ref{thm_gen}.
There exists $k\ge k_0$, such that for any $g_0\in \mathcal{G}^k$, there is $\varepsilon>0$, such that for any two metrics $g_1$, $g_2$ with
$\|g_m-g_0\|_{C^k(M)}\le \varepsilon$, $m=1,2$, we have the following:
\begin{equation}   \label{pro}
\mbox{$d_{g_1}=d_{g_2}$  on $(\partial M)^2$ \ implies  $g_2=\psi_* g_1$}
\end{equation}
with some $C^{k+1}(M)$-diffeomorphism $\psi:M \to M$ \ fixing the boundary.
\end{thm}

\section{Lens rigidity}

%\subsection{\bf Lens Rigidity}

%Suppose we have a Riemannian metric in Euclidean space which is
%the Euclidean metric outside a compact set. The inverse scattering
%problem for metrics is to determine the Riemannian metric by
%measuring the scattering operator (see \cite{Gu}). A similar
%obstruction to the boundary rigidity problem occurs in this case
%with the diffeomorphism $\psi$ equal to the identity outside a
%compact set. It was proven in \cite{Gu} that from the wave front set of
%the scattering
%operator, one can determine, under some conditions on the metric including non-trapping,
%the {\it scattering relation} on the boundary of a
%large ball. This uses high frequency information of the scattering
%operator. In the semiclassical setting Alexandrova \cite{A2} has shown for a
%large class of operators  that the scattering operator associated
%to potential and metric perturbations of the Euclidean Laplacian
%is a semiclassical Fourier integral operator that quantizes  the scattering relation. The scattering relation maps the
%point and direction of a geodesic entering the manifold to the
%point and direction of exit of the geodesic. As mentioned in the previous
%section, the boundary rigidity problem only takes into account the shortest paths. 
For 
non-simple manifolds in particular, if we have conjugate points or the boundary
is not strictly convex, we need to look at the behavior of all the geodesics and the
scattering relation encodes this information. 
We proceed to define
in more detail the scattering relation for non-convex manifolds
and the lens rigidity problem and state our results. 
We note that we will also consider the case of incomplete data, that is when we don't have information
about all the geodesics entering the manifold. More details can be found in \cite{StU4a}, \cite{StU5}.
%Denote by $SM=\left\{(x,\xi)\in  TM; \; |\xi|=1 \right\}$ the unit sphere bundle and set
%\be{a1}
%\partial_\pm  SM = \left\{(x,\xi)\in  \partial SM; \; \pm\langle \nu,\xi\rangle < 0\right\},
%\ee
%where $\nu$ is the unit  interior  normal, $\langle \cdot,\cdot\rangle$ and  stands for the inner product. 

The {scattering relation} 
\be{sig}
\Sigma : \partial_-SM \to  \overline{\partial_+SM}
\ee
is defined by $\Sigma(x,\xi) = (y,\eta) = \Phi^\ell(x,\xi)$, where $\Phi^t$ is the geodesic flow, and $\ell>0$ is the first moment, at which the unit speed geodesic through $(x,\xi)$ hits $\bo$ again. If such an $\ell$ does not exists, we formally set $\ell=\infty$ and  we call the corresponding initial condition and the corresponding geodesic {\it trapping}. This defines also $\ell(x,\xi)$ as a function  $\ell: \partial_-SM \to [0,\infty]$. Note that $\Sigma$ and $\ell$ are not necessarily continuous.

It is convenient to think of $\Sigma$ and $\ell$ as defined on the whole $\partial SM$ with $\Sigma=\Id$ and $\ell=0$ on $\overline{\partial_+SM}$. 
We parametrize the scattering relation in a way that makes it independent of pulling it back by diffeomorphisms fixing $\bo$ pointwise. Let $\kappa_\pm : \partial_\pm  SM \to B(\bo)$ be the orthogonal projection onto the (open) unit ball tangent bundle that extends continuously to the closure of $\partial_\pm  SM$. Then $\kappa_\pm$ are homeomorphisms, and we set 
\be{sig1}
\sigma = \kappa_+\circ \Sigma\circ \kappa^{-1}_- :\overline{B(\bo)} \longrightarrow \overline{B(\bo)}.
\ee
According to our convention, $\sigma=\Id$ on $\partial( \overline{B(\bo)} ) = S(\bo)$. We equip $\overline{B(\bo)}$ with the relative topo\-logy induced by $T(\bo)$, where neighborhoods of boundary points (those in $S(\bo)$) are given by half-neighborhoods, i.e., by neighborhoods in $T(\bo)$ intersected with  $\overline{B(\bo)}$.
It is possible to define $\sigma$ in a way that does not require knowledge of $g|_{T(\bo)}$ by thinking   of any boundary vector $\xi$ as characterized by its angle with $\bo$ and the direction of its tangential projection. 
Let $\D$ be an open subset  of $\overline{B(\bo)}$. A priori, the latter depends on $g|_{T(\bo)}$. By the remark above, we can think of it as independent of $g|_{T(\bo)}$ however.
The {\sl lens rigidity} problem we study  is the following:

\vspace{.05in} 
{\it Given $M$ and $\D$,  do $\sigma$ and  $\ell$, restricted to $\D$, determine $g$ uniquely, up to a pull back of a diffeomorphism that is identity on $\bo$?}
\vspace{.05in}

The answer to this question, even when $\D ={B(\bo)}$, is negative, see \cite{CK}. The known counter-examples are trapping manifolds. 
The boundary rigidity problem and the lens rigidity one are equivalent for simple metrics.

Vargo \cite{V} proved that real-analytic manifolds satisfying an additional mild condition are lens rigid.
Croke has shown that if a manifold is lens rigid, a finite quotient of it
is also lens rigid \cite{Croke04}.  He has also shown that the torus is lens
rigid \cite{Croke_scatteringrigidity}.  Of course the torus is trapping. P.~Stefanov and G. Uhlmann have shown lens rigidity locally
near a generic class of non-simple manifolds \cite{StU5}. In a recent very interesting work, Guillarmou \cite{Colin14} proved that in two dimensions, one can determine from the lens relation the conformal class of the metric if the trapped set is hyperbolic and there are no conjugate points. He also proved  deformation lens rigidity in higher dimensions under the same assumptions.
%, that is the trapped set is hyperbolic and there are no conjugate points. 
%The only result we know for the lens rigidity problem with incomplete (but not local) data
%is for real-analytic metric and metric close to them satisfying some microlocal condition \cite{SU-lens}.

%%%%%%%%%%%%%%%%%%%%%%%%%%%%%%%%%%%%%%%%%%%%%%%%%%%%%%%%%%%%%%%%%

%\subsection{The real-analytic case via boundary determination (result of J. Vargo)}
\subsection{Boundary determination of the jet of $g$}\label{boundary jet}

The lens rigidity in the real-analytic category was studied by Vargo \cite{V}. A key ingredient of the proof is a boundary determination result proved by Stefanov and Uhlmann in \cite{StU5} that we proceed to state. This result shows that one can determine all derivatives of $g$ on $\bo$ from the lens data under some non-conjugacy condition. The theorem is interesting by itself. Notice that $g$ below does not need to be analytic or generic.

\begin{thm}  \label{thm_jet}
Let $(M,g)$ be a compact Riemannian manifold with boundary. Let $(x_0,\xi_0)\in S(\partial M)$ be such that the maximal geodesic $\gamma_{x_0,\xi_0}$ through it is of finite length,  and assume that $x_0$ is not conjugate  to any point in $\gamma_{x_0,\xi_0}\cap \bo$.  If $\sigma$ and $\ell$ are known on some neighborhood of $(x_0,\xi_0)$, then  the jet of $g$ at $x_0$ in boundary normal coordinates is determined uniquely. 
\end{thm}

%%%%%%%%%%%%%%%%%%%%%%%%%%%%%%%%%%%%%%%%%%%%%%%%%%%%%%%%%%%%%%%%%%%%%%%%

\subsection{\bf The microlocal condition}\label{main assumption}

 To state the results of \cite{StU4} and \cite{StU5} we need some definitions. 
%\subsubsection{Main assumptions}  \label{sec_3.1}
\begin{Def}   \label{def_comp} 
We say that $\D$ is {\em complete} for the metric $g$, if for any $(z,\zeta)\in T^*M$ there exists a maximal in $M$, finite length unit speed geodesic $\gamma :[0,l] \to M$ through $z$, normal to $\zeta$, such that 
\begin{align}  \label{i}
&\left\{(\gamma(t), \dot\gamma(t)); \;0\le t\le l  \right\}   \cap S(\bo)  \subset  \D,\\ \label{ii}
&\text{there are no conjugate  points on $\gamma$.}
\end{align}
\noindent We call the $C^k$ metric $g$ {\em regular}, if a complete set $\D$ exists, i.e., if $\overline{B(\bo)}$ is complete.
\end{Def}
If $z\in\bo$ and $\zeta$ is conormal to $\bo$, then $\gamma$ may reduce to one point. 

\noindent\textbf{Topological Condition (T):} Any path in $M$ connecting two boundary points  is homotopic to a polygon $c_1\cup \gamma_1 \cup c_2\cup\gamma_2\cup\dots \cup\gamma_k \cup c_{k+1}$  with the properties that for any $j$,
(i) $c_j$ is a path on $\bo$;
(ii) $\gamma_j :[0,l_j]\to M$ is a  geodesic lying  in $\Mint$ with the exception of its endpoints and is transversal to $\bo$ at both ends; moreover, $\kappa_-(\gamma_j(0), \dot\gamma_j(0)) \in  \D$; 
\medskip
Notice that (T) is an open condition w.r.t.\ $g$, i.e., it is preserved under small $C^2$ perturbations of $g$. 
To define the $C^K(M)$ norm below in a unique way, we choose and fix a finite atlas on $M$.

\subsubsection{Results about tensor tomography} We refer to \cite{StU4a} for more details about the results in this section. 
It turns out that a linearization of the lens rigidity problem is again the problem of s-injectivity of the ray transform $I$. Here and below we sometimes drop the subscript $g$. 
Given $\D$ as above, we denote by $I_\D$ (or $I_{g,\D}$) the ray transform $I$ restricted to the maximal geodesics issued from $(x,\xi)\in \kappa_-^{-1}(\D)$. 
The first result of this section generalizes Theorem~\ref{thm_an}.
\begin{thm} \label{thm_an2} \ 
Let $g$ be an analytic, regular metric on $M$. Let $\D$ be  complete and open. Then $I_\D$ is s-injective.
\end{thm}
The theorem above allows us to formulate a generic result:
\begin{thm}  \label{thm_2}
Let $\mathcal{G}\subset C^k(M)$ be an open set of regular Riemannian metrics on $M$ such that (T) is satisfied for each one of them. Let the  set $\D' \subset \partial SM$ be open and complete for each $g\in\mathcal{G}$. Then there exists  an open and dense 
subset  $\mathcal{G}_s$  of $\mathcal{G}$ such that $I_{g,\D'}$ is s-injective for any $g\in \mathcal{G}_s$.
\end{thm}
Of course, the set $\mathcal{G}_s$ includes all real analytic metrics in $\mathcal{G}$. 
\begin{cor}  \label{cor_1} 
Let $\mathcal{R}(M)$ be the set of all regular $C^k$ metrics on $M$ satisfying (T) equipped with the $C^k(M_1)$ topology. Then for $k\gg1$, the subset of metrics for which the   X-ray transform $I$ over all simple geodesics through all points in $M$ is s-injective, is open and dense in  $\mathcal{R}(M)$.
\end{cor}

\subsubsection{Results about the non-linear lens rigidity problem} Using the results above, it was proven in \cite{StU5}  the following about the lens rigidity problem on manifolds satisfying the assumptions at the beginning of Section~\ref{main assumption}. More details can be found in 
\cite{StU5}.
Theorem~\ref{thm_1a} below says, loosely speaking, that for the classes of manifolds and metrics we study, the uniqueness question for the non-linear lens rigidity problem can be answered locally by linearization. This is a non-trivial implicit function type of theorem however because our success heavily depends on the a priori stability estimate that the s-injectivity of $I_\D$ implies; %see  Theorem~\ref{thm_stab_l};  
and the latter is based on the hypoelliptic properties of $I_\D$. We work with two metrics $g$ and $\hat g$; and will denote objects related to $\hat g$ by $\hat \sigma$, $\hat \ell$, etc. 
\begin{thm}  \label{thm_1a}
Let $(M,g_0)$ satisfy the topological assumption (T), with $g_0\in C^k(M)$  a regular Riemannian metric  with $k\gg1$. Let $\D$ be open and complete for $g_0$, and assume that there exists $\D'\Subset\D$ so that $I_{g_0,\D'}$ is s-injective. Then there exists $\eps>0$, such that for any two metrics $g$, $\hat g$ satisfying  
\be{thm3}
\|g-g_0\|_{C^k(M)}+ \|\hat g-g_0\|_{C^k(M)}\le\eps,
\ee
the relations 
$$
\sigma = \hat\sigma, \quad \ell = \hat\ell \quad \text{on $\D$}
$$
imply that there is a $C^{k+1}$ diffeomorphism $\psi :M \to M$ fixing the boundary such that
$$
\hat g=\psi^*g.
$$
\end{thm}
By Theorem~\ref{thm_2}, the requirement that $I_{g_0,\D'}$ is \mbox{s-injective} is a generic one for $g_0$. Therefore, 
Theorems~\ref{thm_1a} and \ref{thm_2} combined imply that there is local uniqueness, up to isometry, near a generic set of regular metrics. 
\begin{cor}  \label{cor_1a}
Let $\D'\Subset \D$, $\mathcal{G}$, $\mathcal{G}_s$ be as in Theorem~\ref{thm_2}. Then the conclusion of Theorem~\ref{thm_1a} holds for any $g_0\in\mathcal{G}_s$. 
\end{cor}

Bao and Zhang have proved in \cite{Bao-Zhang} a stability estimate for the lens rigidity problem under some microlocal conditions.

%%%%%%%%%%%%%%%%%%%%%%%%%%%%%%%%%%%%%%%%%%%%%%%%%%%%%%%%%%%%%%%%%%%%%%%%%%%%%
%%%%%%%%%%%%%%%%%%%%%%%%%%%%%%%%%%%%%%%%%%%%%%%%%%%%%%%%%%%%%%%%%%%
%%%%%%%%%%%%%%%%%%%%%%%%%%%%%%%%%%%%%%%%%%%%%%%%%%%%%%%%%%%%%%%%%%%%%%%%%%%

\bigskip

\section{Boundary and lens rigidity with partial data}

Now we consider the boundary rigidity problem with partial (local) data, that is, we know the boundary 
distance function for points on the boundary near a given point. Partial data problems arise naturally in applications since in many cases one doesn't have access to the whole boundary. We first study the corresponding linearized problem with partial data in section \ref{linear conformal}--\ref{linearized general}. In section \ref{nonlinear conformal} we consider the partial data boundary and lens rigidity problem through a pseudolinearization. Section \ref{global foliation} is devoted to the applications of the various partial data results to global problems under a geometric foliation condition.

\subsection{The linearized problem for conformal metrics}\label{linear conformal}

It is well-known that the linearization of the boundary rigidity and lens rigidity problem is the geodesic ray transform of symmetric 2-tensors. In particular, if one restricts the linearization in a fixed conformal class, it reduces to the geodesic ray transform of smooth functions on the manifold. In this section, we discuss the {\em local} geodesic ray transform of functions in dimension $\geq 3$.

Let $X$ be a strictly convex domain in a Riemannian manifold $(\tilde X,g)$ of dimension $\geq 3$ with boundary defining function $\rho$ (so $\rho\in\CI(\tilde X)$, $\rho>0$ in $X$, $<0$ on $\tilde X\setminus \overline{X}$, vanishes non-degenerately at $\pa X$). We recall that {\em strict convexity} means that geodesics which are tangent to $\pa X$ are {\em only} simply tangent, curving away from $X$, or more explicitly in terms of Hamiltonian dynamics, with $G$ the dual metric function on $T^*\tilde X$, if for some $\beta\in T^*_p \tilde X$, $p\in\pa X$, $\beta\neq 0$, one has $(H_G\rho)(\beta)=0$ where $H_G$ is the Hamiltonian vector field associated with $G$, then necessarily $(H_G^2\rho)(\beta)<0$. %In this paper we consider the local inverse problem for the geodesic X-ray transform. That is, 
For an open set $O\subset \overline{X}$, we call geodesic segments $\gamma$ of $g$ which are contained in $O$ with endpoints at $\pa X$ {\em $O$-local geodesics}; we denote the set of these by $\cM_O$. Thus, $\cM_O$ is an open subset of the set of all geodesics, $\cM$. We then define the {\em local geodesic transform} of a function $f$ defined on $X$ as the collection $(If)(\gamma)$ of integrals of $f$ along geodesics $\gamma\in \cM_O$, i.e.\ as the restriction of the X-ray transform to $\cM_O$.
 
The main result of this section is an invertibility result by Uhlmann and Vasy \cite{UV2} for the local geodesic transform on neighborhoods of $p\in \partial X$ in $\overline{X}$ of the form $\{\tilde x>-c\}$, $c>0$, where $\tilde x$ is a function with $\tilde x(p)=0$, $d\tilde x(p)=-d\rho(p)$, see Figure~\ref{fig:convex-1} below.

\begin{thm}\label{local geodesic function}
For each $p\in\pa X$, there exists a function $\tilde x\in\CI(\tilde X)$ vanishing at $p$ and with $d\tilde x(p)=-d\rho(p)$ such that for $c>0$ sufficiently small, and with $O_p=\{\tilde x>-c\}\cap\overline{X}$, the local geodesic transform is injective on $H^s(O_p)$, $s\geq 0$. Further, let $H^s(\cM_{O_p})$ denote the restriction of elements of $H^s(\cM)$ to $\cM_{O_p}$, and for $\digamma>0$ let
$$
H^s_\digamma(O_p)=e^{\digamma/(\tilde x+c)} H^s=\{f\in H^s_{\loc}(O_p):\ e^{-\digamma/(\tilde x+c)} f\in H^s(O_p)\}.
$$
Then for $s\geq 0$ there exists $C>0$ such that for all $f\in H^s_\digamma(O_p)$,
$$
\|f\|_{H^{s-1}_\digamma(O_p)}\leq C\|If|_{\cM_{O_p}}\|_{H^s(\cM_{O_p})}.
$$
\end{thm}

\begin{rem*}
Here the constant $C$ is uniform in $c$ for small $c$, and indeed if we consider the regions $\{\rho\geq\rho_0\}\cap\{\tilde x>-c\}$ with $|\rho_0|$ and $|c|$ sufficiently small and such that this intersection is non-empty, the estimate is uniform in both $c$ and $\rho_0$. Further, the estimate is also stable under sufficiently small perturbations of the metric $g$, i.e.\ the constant is uniform. (Notice that the hypotheses of the theorem are satisfied for small perturbations of $g$)
\end{rem*}

We remark that for this result one only needs to assume convexity near the point $p.$ This local result is new even in the case that the metric is conformal to the Euclidean metric. We also point out that we also get a reconstruction method in the form of a Neumann series. See \cite{UV2} for more details. These results generalize support type theorems to the smooth case for the geodesic X-ray transform given in \cite{Krish1} for simple real-analytic metrics.

%While this large weight $e^{\digamma/(\tilde x+c)}$ means that the control over $f$ in terms of $If$ is weak at $\tilde x=-c$, the control is uniform in compact subsets of $O_p$: these weights are bounded below on $O_p$ by a positive constant, and bounded above on compact subsets of $O_p$ (in particular at parts of $\pa X$). Here $\digamma>0$ can be taken small, but not vanishing. Further, $\tilde x$, whose existence is guaranteed by the theorem, is such that $\tilde x=-c$ is concave from the side of $O_p$.

For the rest of this section, we give a sketch of the proof of Theorem \ref{local geodesic function}. In order to motivate the proof, recall that Stefanov and Uhlmann \cite{StU4a} have
shown that under a microlocal condition on the geodesics, one can
recover the singularities of functions from their X-ray transform, and
indeed from a partial X-ray transform (where only some geodesics are
included in the X-ray family $\cM'$). (In
fact, they also showed analogous statements for the transforms on
tensors.) Roughly speaking what one needs is that given a covector
$\nu=(z,\zeta)$, one needs to have a geodesic in $\cM'$ normal to $\zeta$ at $z$
such that in a neighborhood of $\nu$ a simplicity condition is
satisfied. Indeed, under these assumptions, a microlocal version of
the normal operator, $(QI)^*(QI)$, where $Q$ microlocalizes to $\cM'$
roughly speaking, is an elliptic pseudodifferential operator.

Now, in
dimension $\geq 3$, if the boundary $\pa X$ is strictly convex,
one can use geodesics which are almost tangent to
$\pa X$ to give a family $\cM'$ which satisfies the above conditions
for $\nu$ with $z$ near $\pa X$. Concretely, let $\rho$ be a boundary defining function of $X$, i.e.\
$\rho>0$ in $X$, $\rho=0$ at $\pa X$, and $d\rho\neq 0$ at $\pa X$; we
assume that in fact $\rho$ is defined on the ambient space $\tilde X$
as above.
First we choose an initial neighborhood
$U$ of $p$ in $\tilde X$ and a function $\tilde x$ defined on it with $\tilde x(p)=0$, $d\tilde x(p)=-d\rho(p)$, $d\tilde x\neq 0$ on $U$ with convex level sets from the side
of the sublevel sets and
such that $O_c=\{\tilde x>-c\}\cap\{\rho\geq 0\}$ satisfies
$\overline{O_c}\subset U$ is compact. One example of such $\tilde x$ is 
$$
\tilde x(z)=-\rho(z)-\ep|z-p|^2,
$$
whose level sets are slightly
less convex, see Figure~\ref{fig:convex-1} below.

\begin{figure}[ht]\label{fig:convex-1}
\begin{center}\includegraphics[width=80mm]{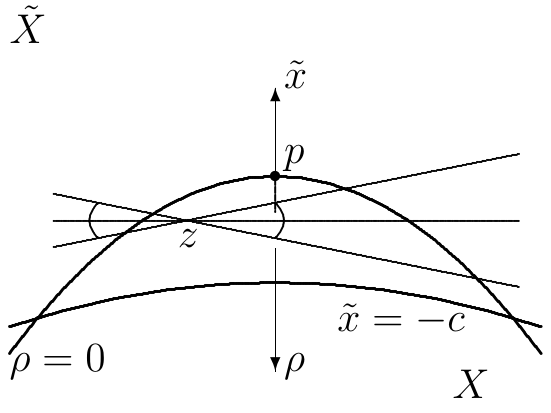}\end{center}
\caption{The functions $\rho$ and $\tilde x$ when the background is flat space $\tilde X$. The intersection of $\rho\geq 0$ and $x_c> 0$ (where $x_c=\tilde x+c$, so this is the region $\tilde x>-c$) is the lens shaped region $O_p$. Note that, as viewed from the superlevel sets, thus from $O_p$, $\tilde x$ has concave level sets. At the point $z$, one only considers integrals over geodesics in the indicated small angle. As $z$ moves to the artificial boundary $x_c=0$, the angle of this cone shrinks like $C x_c$ so that in the limit the geodesics taken into account become tangent to $x_c=0$.}
\end{figure}

We consider geodesics $\gamma_\nu:I\to \tilde X$ parameterized by
$\nu=(z,\zeta)\in S \tilde X$ (the sphere bundle of $\tilde X$ realized as a subbundle
of $T\tilde X$, e.g.\ via a Riemannian metric; we actually use a
slightly modified identification) with $\gamma_\nu'(0)=\nu$, if $\nu$ is tangent to a level set of $\tilde x$ in
$O_c$, i.e. if $\frac{d}{dt}(\tilde x\circ\gamma_\nu)|_{t=0}=0$,
then $\alpha(\nu):=\frac{d^2}{dt^2}(\tilde x\circ\gamma_\nu)|_{t=0}\geq C>0$. The
lower bound on the second derivative is a concavity statement for the level sets of
$\tilde x$
from the side of the superlevel sets. Let $x=x_c=\tilde x+c$ as above. Thus,
$x$ is a boundary defining function for $\{\tilde x>-c\}$; for the
time being we regard $c$ as fixed. A
consequence of the uniform concavity statement is that, with
$\lambda=\frac{d}{dt}(x\circ\gamma_\nu)|_{t=0}$, if $C_1>0$ is
sufficiently small and $|\lambda|<C_1 \sqrt{x}$, then $\gamma_\nu$ remains in $x\geq 0$. Rather than using
this range of $\lambda$, we instead use the stronger bound
$|\lambda|< C_2 x$, and define $A$, which is
essentially a `microlocal normal operator' for the geodesic ray
transform, to be an average:
$$
Af(z)=x^{-2}\int If(\gamma_\nu)\chi(\lambda/x)\,d\mu(\nu),
$$
where $\mu$ is a non-degenerate smooth measure on $S\tilde X$, and $\chi\geq 0$ has
compact support. Then it is routine to check that the principal symbol of $A$ at $(z,\zeta)$ is a multiple of 
$$|\zeta|^{-1}\int_{|\hat Z|=1,\,\hat Z\perp \zeta} \chi(z,\hat Z)\,\sigma(z,\hat Z)\,d\hat Z$$
for a positive density $\sigma$. As observed in \cite{StU4a} this gives a recovery of
singularities for the local problem we are considering, it yields no
invertibility or reconstruction. Indeed for the latter we would like
to have an invertible operator on a space of
functions on $O_c$; in particular, as one approaches $x=0$ one
would need to only allow integrals over geodesics in a narrow cone,
becoming tangent to $x=0$, which takes one outside the
framework of standard pseudodifferential operators.

To remedy this, we introduce the artificial boundary $x=0$,
and work with pseudodifferential operators in $x>0$ which
degenerate at $x=0$. The
particular degeneration we end up with is Melrose's scattering
calculus \cite{RBMSpec}. This is defined on manifolds with
boundary, with boundary defining function $x$, and is based on
degenerate vector fields $x^2\pa_x$ and $x\pa_{y_j}$, where the $(x,y_1,\ldots,y_{n-1})$
are local coordinates. This has the effect of pushing $x=0$ `to
infinity' (these vector fields are complete under the exponential
map). %Thus, ultimately, our approach is based on working in a framework with an artificial boundary which is effectively `at infinity', and we work with function spaces allowing exponential growth at this boundary. Thus the control at $x=0$ will be quite weak in a sense, though one has the standard control when $x$ is bounded away from $0$. Since $x=0$ is just an artificial boundary, this is a satisfactory situation.
In particular, if we consider $\overline{\mathbb R^n}$ the radial compactification of $\mathbb R^n$, let $\Psisc$ stands for the scattering calculus of Melrose, $\Psisc^{m,l}(\overline{\mathbb R^n})$ corresponds to symbols $a\in S^{m,l}$ satisfying
$$|(D_z^\alpha D_\zeta^\beta a)(z,\zeta)|\leq C_{\alpha \beta}\langle z\rangle^{l-|\alpha|}\langle\zeta\rangle^{m-|\beta|}.$$

To describe the Schwartz kernel of a scattering pseudodifferential operator, it is convenient to introduce the scattering coordinates
$$x,\,y,\, X=\frac{x'-x}{x^2},\, Y=\frac{y'-y}{x}$$
valid for $x>0$. Then the Schwartz kernel of an element of $\Psisc^{m,l}(\{x\geq 0\})$ is of the form
$K=x^{-l}\tilde K(x,y,X,Y)$, where $\tilde K$ is smooth in $(x,y)$ down to
$x=0$ with values in conormal distributions on $\RR^n_{X,Y}$, conormal
to $\{X=0,\ Y=0\}$, which are Schwartz at infinity (i.e.\ decay
rapidly at infinity with all derivatives). Notice that the operator $A$ does not have rapid decay as $(X,Y)\to \infty$, to remedy this we introduce exponential weights and consider for $\digamma\in\RR$,
$$
A_\digamma=e^{-\digamma/x}A e^{\digamma/x}.
$$
It is not difficult to check that under scattering coordinates the Schwartz kernel of $A_F$ takes the form 
\begin{equation}\label{A_digamma kernel}
e^{-\digamma X/(1+xX)}\chi\Big(\frac{X-\alpha|Y|^2}{|Y|}+O(x)\Big) |Y|^{-n+1}J,
\end{equation}
where the density factor $J$ is smooth and positive with $J|_{x=0}\equiv 1$. Thus the following result holds for $A_\digamma$, of which the ellipticity statement is essentially for the same reason as that for operator $A$.

\begin{lemma}\label{interior elliptic}
For $\digamma>0$, $A_\digamma$ is in $\Psisc^{-1,0}(\{x\geq 0\})$, and is elliptic in the sense that the standard principal
symbol is such near the boundary (up to the boundary, $x=0$).
\end{lemma}

However, even when this holds
globally on a compact space, this ellipticity is not sufficient
for Fredholm properties (between Sobolev spaces of order shifted by
$1$), or the corresponding estimates, due to
the boundary $x=0$. In general, scattering pseudodifferential
operators also have a principal symbol at the boundary, which is a
(typically non-homogeneous)
function on a cotangent bundle; this needs to be invertible (non-zero) globally
to imply Fredholm properties. Similarly, estimates implying the finite dimensionality
of localized (in $O$) non-trivial nullspace as well as stability estimates,
follow
if this principal symbol is also invertible on $O$. (Note that here
localization {\em does allow} the support in $\{x\geq 0\}$ to include
points at $x=0$!)

The main technical result is the following
\begin{lemma}\label{boundary elliptic}
For $\digamma>0$ there exists $\chi\in\CI_c(\RR)$, $\chi\geq 0$,
$\chi(0)=1$, such that for the corresponding operator 
$A_\digamma=e^{-\digamma/x}Ae^{\digamma/x}$ the boundary symbol is 
elliptic; indeed, this holds for all $\chi$ sufficiently close in 
Schwartz space to a 
specific Gaussian. 
\end{lemma}

\noindent{\em Sketch of the proof:} By \eqref{A_digamma kernel} the restriction of the Schwartz kernel at $x=0$ is
\begin{equation}\label{A_digamma boundary kernel}
\tilde K(y,X,Y)=e^{-\digamma X}|Y|^{-n+1}\chi\Big(\frac{X-\alpha(0,y,0,\hat Y)|Y|^2}{|Y|}\Big), 
\end{equation}
where $\hat Y=Y/|Y|$, the desired almost invertibility (up to compact errors) amounts to the Fourier
transformed kernel, $\mathcal F_{X,Y} \tilde K(y,.,.)$ being bounded below in absolute value by
$c\langle(\xi,\eta)\rangle^{-1}$, $c>0$ (here $(\xi,\eta)$ are the
Fourier dual variables of $(X,Y)$).

In order to find a suitable $\chi$, we first make
a slightly inadmissible choice for an easier computation,
namely we take $\chi(s)=e^{-s^2/(2\digamma^{-1}\alpha)}$, so
$\hat\chi(.)=c\sqrt{\digamma^{-1}\alpha}e^{-\digamma^{-1}\alpha|.|^2/2}$ for appropriate $c>0$. Thus,
$\chi$ does not have compact support, and an approximation argument
will be necessary at the end.

As mentioned above, $\alpha$ (as the Hessian of $x$) restricting on the tangent plans of the level sets of $x$ has a positive lower bound. Moreover, in the case of geodesics $\alpha(x,y,0,\hat Y)$ is a positive definite quadratic form in $\hat Y$, thus one has $\alpha |Y|^2=Q(Y,Y)$, a quadratic form in $Y$. Writing $Q^{-1}(Y,Y)$ for the dual quadratic form, we have that the $X$-Fourier transform of $\tilde K$ is a none multiple of
\begin{equation}\label{X-FT}
\digamma^{-1/2}\sqrt{\alpha}|Y|^{2-n}e^{-\digamma^{-1}(\xi^2+\digamma^2)Q(Y,Y)/2},
\end{equation}
where the last factor is a real Gaussian, thus is Schwartz in $Y$ for $\digamma>0$. Taking into account the formula of the Fourier transform of Gaussian functions, the $Y$-Fourier transform of \eqref{X-FT} is virtually a positive multiple of
\begin{equation}\label{Y-FT}
\langle\xi\rangle^{-1}\varphi(\eta/\langle\xi\rangle)=|\eta|^{-1}|\eta'|\varphi(\eta')=|\eta|^{-1}\tilde\varphi(\langle\xi\rangle/|\eta|,\eta'/|\eta'|),
\end{equation}
where $\langle\xi\rangle=(\xi^2+\digamma^2)^{1/2}$, $\eta'=\eta/\langle\xi\rangle$, with $\varphi$ a positive $0$-th order symbol near $0$ and $\tilde\varphi$ smooth positive near $0$ in the first argument. This assures lower bounds
$c\langle(\xi,\eta)\rangle^{-1}$, $c>0$, i.e.\ elliptic lower
bounds. 
\qed

Lemma \ref{boundary elliptic} implies that
$A_\digamma$ is Fredholm on proper function spaces, i.e.\ $A$ itself is Fredholm on
exponentially weighted spaces, where {\em exponential growth} is
allowed at $x=0$. We now recall that $x=x_c$ depends on $c$, with all
estimates uniform for $c$ remaining in a compact set, and the
argument is finished by showing that for $c>0$ sufficiently small one
not only has Fredholm properties but also invertibility, essentially
as the Schwartz kernel has small support. This proves the main theorem.

We remark that J. Boman has given in
\cite{Bom}  counterexamples for local uniqueness for the X-ray
transform on a plane that integrates along lines with a dense family of smooth
weights so that we expect some restrictions on the family of curves for the uniqueness for the local X-ray transform in dimension two. 

%\bigskip

%The main theorem is proved by considering an operator $A$ . 

% We define an operator $L$ which integrates $If$ over a subset of $\cM_{O_c}$ with a $\CI$ cutoff, and consider $A=L\circ I$. We consider this operator as a map between appropriate function spaces on $O_c$. It turns out that with the subset of geodesics we choose, the exponential conjugate $A_\digamma$ of $A$ is a pseudodifferential operator in Melrose's scattering calculus \cite{RBMSpec}. (The exponential conjugate corresponds to working with exponentially weighted spaces for $A$.) We show that $A_\digamma$ is a Fredholm operator, and indeed that it is invertible for $c$ near $0$.

%We thus show that in the case when $\frac{d^2}{dt^2}(\tilde x\circ\gamma_\nu)|_{t=0}$ is a quadratic form in $\zeta$ subject to $\lambda=0$, which is the case with geodesics, for suitable choices of $\chi$, namely essentially cutoff Gaussians, this principal symbol is invertible when the weight $\digamma$ satisfies $\digamma>0$. 

%%%%%%%%%%%%%%%%%%%%%%%%%%%%%%%%%%%%%%%%%%%%%%%%%%%%%%%%%%%%%%%%%%%%%%%%%%%%%%%%%%%%%%%%%%%%%%%5

\subsection{More general curves}\label{general curve partial data}

We note that the geodesic nature of the curves was only used in the crucial step of showing that the principal symbol at the boundary is invertible. %While our argument relied on properties of the geodesics to analyze this symbol, subsequently to our paper, Hanming Zhou succeeded in analyzing it in general and proved the result for more general families of curves.
The result of the local invertibility of the geodesic ray transform, Theorem \ref{local geodesic function}, was extended by Zhou \cite[Appendix]{UV2} to the X-ray transform on more general curves.

Given a Riemannian manifold $(\tilde{X}, g)$ of dimension $\geq 3$, we consider smooth curves $\gamma$ on $\tilde{X}$, $|\dot{\gamma}|\neq 0$, that satisfy the following equation
\begin{equation}\label{general curves 1}
\nabla_{\dot{\gamma}}\dot{\gamma}=G(\gamma, \dot{\gamma}),
\end{equation}
where $\nabla$ is the Levi-Civita connection, $G(z, v)\in T_z\tilde{X}$ is smooth on $T\tilde{X}$. $\gamma=\gamma_{z,v}$ depends smoothly on $(z,v)=(\gamma(0),\dot{\gamma}(0))$. We call the collection of such smooth curves on $\tilde{X}$, denoted by $\mathcal{G}$, {\it a general family of curves}. For the sake of simplicity, we assume $\gamma\in \mathcal{G}$ are parameterized by arclength (one can always reparametrize the curve to make this happen, and we will see later that our method also works for curves with non-constant speed). Note that if $G\equiv 0$, $\mathcal{G}$ is the set of usual geodesics; if $G$ is the Lorentz force corresponding to some magnetic field, then $\mathcal{G}$ consists of magnetic geodesics. We consider the X-ray transform of smooth functions along a general family of curves, i.e. $(If)(\gamma), \gamma\in\mathcal{G}.$ X-ray transforms for general curves were studied in e.g. \cite{DPSU, FSU}.

Let $X$ be a domain in $\tilde{X}$ with boundary defining function $\rho$, and $p\in \partial X$ a boundary point. We say that $X$ (or $\partial X$) is {\it strictly convex (concave)} at $p$ with respect to $\mathcal{G}$ if for any $\gamma\in \mathcal{G}$ with $\gamma(0)=p$, $\dot{\gamma}(0)=v\in T_p(\partial X)$, we have $\frac{d^2}{dt^2}\rho(\gamma(t))|_{t=0}<0\,(> 0)$. It is easy to see that the geometric meaning of our definition is similar to the usual convexity with respect to the metric (geodesics).

Now assume $X$ is strictly convex at $p\in \partial X$ with respect to $\mathcal{G}$. Similar to the settings in Section \ref{linear conformal}, we obtain a smooth function $x$ whose level sets are strictly concave with respect to $\mathcal{G}$ from the super-level sets of $x$. In particular, if $\alpha$ is defined the same as in Section \ref{linear conformal} as the Hessian of $x$ with respect to $\mathcal{G}$, then $\alpha(x,y,0,\hat Y)>0$, i.e. $\alpha$ is positive on the tangent plane of the level sets of $x$. It is known that $\alpha(x,y,0,\hat Y)$ defines a positive definite quadratic form for the usual geodesics, however for a general family of curves, it no longer has such special structure.

We consider the operator $A_\digamma=e^{-\digamma/x}Ae^{\digamma/x}$ introduced in Section \ref{linear conformal}. Similar to Lemma \ref{interior elliptic} $A_\digamma$ is in $ \Psi^{-1,0}_{sc}(\{x\geq 0\})$ for $\digamma>0$, with elliptic standard principal symbol. %In particular, the Schwartz kernel of $A_\digamma$ at the scattering front face $x=0$, $\tilde{K}(y,X,Y)$, is given in Lemma 3.5. 
Now to show the invertibility of $A_\digamma$, it suffices to establish the following lemma, analogous to Lemma \ref{boundary elliptic}, for a general family of curves. %we need to verify the ellipticity of its boundary principal symbol. The proof of Lemma \ref{lemma:elliptic} relies on the quadratic structure of $\alpha(x,y,0,\omega)$ which is not available for general curves, so we analyze the principal symbol in a different way, and this method works for a general family of curves. 

\begin{lemma}\label{boundary elliptic general}
For $\digamma >0$ there exists $\chi\in C^{\infty}_c(\mathbb{R})$, $\chi\geq 0, \chi(0)=1$, such that the boundary principal symbol of corresponding $A_\digamma$ is elliptic.
\end{lemma}

\noindent{\em Sketch of the proof:} Similar to the strategy in the proof of Lemma \ref{boundary elliptic}, we first analyze the boundary principal symbol for the case $\chi(s)=e^{-s^2/(2\digamma^{-1}\alpha)}$, a compact supported one follows by approximation. It is not difficult to see that the Schwartz kernel of $A_\digamma$ and its boundary restriction $\tilde K$ on $x=0$ have the same type as \eqref{A_digamma kernel} and \eqref{A_digamma boundary kernel} respectively, except that the values of $\alpha$ at $\hat Y$ and $-\hat Y$ are not equal in general. So the $X$-Fourier transform of $\tilde K$ is a non-zero multiple of 
\begin{equation}\label{X-FT general}
|Y|^{2-n}\Big((\digamma^{-1}\alpha_+)^{\frac{1}{2}}e^{-\digamma^{-1}(\xi^2+\digamma^2)\alpha_+|Y|^2/2}+(\digamma^{-1}\alpha_-)^{\frac{1}{2}}e^{-\digamma^{-1}(\xi^2+\digamma^2)\alpha_-|Y|^2/2}\Big),
\end{equation}
where $\alpha_+=\alpha(0,y,0,\hat Y), \, \alpha_-=\alpha(0,y,0,-\hat Y)$.

As mentioned previously, in general $\alpha$ is not a quadratic form in $\hat Y$, which means the exponential term in \eqref{X-FT general} is not Gaussian in $Y$, thus we use polar coordinates to compute the $Y$-Fourier transform of \eqref{X-FT general}. We denote $\frac{\digamma^{-1}(\xi^2+\digamma^2)}{2}$ by $b$, then the boundary principal symbol is a constant multiple of
\begin{align*}
& \int_0^{+\infty}\int_{\mathbb S^{n-2}} e^{-i\eta\cdot\hat{Y}|Y|}|Y|^{2-n}(\alpha_+^{1/2}e^{-b\alpha_+|Y|^2}+\alpha_-^{1/2}e^{-b\alpha_-|Y|^2})|Y|^{n-2}\,d|Y|d\hat Y\\
=& c \int_{\mathbb S^{n-2}} b^{-1/2}(e^{-|\eta\cdot\hat Y|^2/4b\alpha_+}+e^{-|\eta\cdot\hat Y|^2/4b\alpha_-})\,d\hat Y\\
= & c' \langle\xi\rangle^{-1}\int_ {\mathbb S^{n-2}} e^{-|\frac{\eta}{\langle\xi\rangle}\cdot\hat Y|^2/4c\alpha(0,y,0,\hat Y)}\,d\hat Y=c' |\eta|^{-1}\int_ {\mathbb S^{n-2}} \frac{|\eta|}{\langle\xi\rangle} e^{-|\frac{\eta}{\langle\xi\rangle}\cdot\hat Y|^2/4c\alpha(0,y,0,\hat Y)}\,d\hat Y.
\end{align*}
Here 
$$\int_ {\mathbb S^{n-2}} e^{-|\frac{\eta}{\langle\xi\rangle}\cdot\hat Y|^2/4c\alpha(0,y,0,\hat Y)}\,d\hat Y\quad \mbox{and} \quad \int_ {\mathbb S^{n-2}} \frac{|\eta|}{\langle\xi\rangle} e^{-|\frac{\eta}{\langle\xi\rangle}\cdot\hat Y|^2/4c\alpha(0,y,0,\hat Y)}\,d\hat Y$$
play the roles of $\varphi$ and $\tilde\varphi$ from \eqref{Y-FT}. Now taking into account the positivity of $\alpha$, this gives the desired lower bound of the boundary symbol.
\qed

We denote the set of $O$-local curves with respect to $\mathcal{G}$ by $\mathcal{G}_O$, as a consequence of Lemma \ref{boundary elliptic general}, the following local invertibility result holds for a general family of curves.

\begin{thm}
Assume $X$ is strictly convex at $p\in \partial X$ with respect to a general family of curves $\mathcal{G}$, with $O_p=\{x>0\}\cap\overline X$, then the local X-ray transform for $\mathcal{G}_{O_p}$ is injective on $H^s(O_p)$, $s\geq 0$ with the stability estimate
$$\|f\|_{H_\digamma^{s-1}(O_p)}\leq C\|If|_{\mathcal G_{O_p}}\|_{H^s(\mathcal G_{O_p})}.$$
\end{thm}

%Similar to the geodesic case, if $X$ has compact closure and can be foliated by hypersurfaces that are strictly convex with respect to $\mathcal G$, then the corresponding global X-ray transform is injective. Previously, injectivity result of the global X-ray transform for a general family of curves was only known in the real analytic category \cite{FSU}.

\begin{rem*}
If we add a non-vanishing weight $w\in C^{\infty}(T\tilde X)$ to the X-ray transform, i.e.
$$(I_wf)(\gamma)=\int w(\gamma(t),\dot{\gamma}(t)) f(\gamma(t))\,dt,$$
a slight modification of the proof of Lemma \ref{boundary elliptic general} allows one to conclude the local invertibility of $I_w$ for a general family of curves. Notice that given a curve $\gamma$ with $|\dot{\gamma}|\neq 0$, a reparametrization exactly introduces a non-vanishing weight to the integral, thus local invertibility of X-ray transform also holds for general families of curves with non-constant speed.
\end{rem*}

%%%%%%%%%%%%%%%%%%%%%%%%%%%%%%%%%%%%%%%%%%%%%%%%%%%%%%%%%%%%%%%%%%%%%%%%%%%%%%%%%%%%%%%%

\subsection{Linearized problem in general}\label{linearized general}

In this section we discuss the linearized problem in general, i.e. the geodesic ray transform of symmetric tensor fields. Let $(M,g)$ be a compact Riemannian  manifold with boundary. %The X-ray transform of symmetric covector fields of order $m$ is given by 
%\be{1}
%If(\gamma) = \int \langle f(\gamma(t)), \dot\gamma^m(t) \rangle \, d t,
%\ee
%where, in local coordinates, $\langle f, v^m\rangle  =f_{i_1\dots i_m} v^{i_1}\dots v^{i_m}$, and $\gamma$ runs over all  (finite length) geodesics with endpoints on $\partial M$.
The problem we study is the invertibility of $I$. It is well known that \textit{potential} vector fields,  i.e., $f$ which are a symmetric differential $d^sv$ of a symmetric field of order $m-1$ vanishing on $\partial M$ (when $m\ge1$), are in the kernel of $I$. We prove the local invertibility, up to potential fields,  and stability of the geodesic X-ray transform on tensor fields of order $1$ and $2$ near a strictly convex boundary point, on manifolds with boundary of dimension $n\ge3$. We study the $m=1,\,2$ cases for simplicity of the exposition but the methods extend to any $m\ge 1$.

We set this up in the same way as
in Section \ref{local geodesic function} by considering a function $\tilde x$ with
strictly concave level sets from the super-level set side for levels
$c$, $|c|<c_0$, and letting
$$
x_c=\tilde x+c,\ \Omega_c=\{x_c\geq 0,\ \rho\geq 0\}.
$$ 
The main result is the following, see \cite{SUV} for more details.

\begin{thm}\label{thm:local-linear-intro-1}%(See Corollaries~\ref{cor:local-linear-one-form}-\ref{cor:local-linear-2-tensor}.)
With $\Omega=\Omega_c$ as above, there is $c_0>0$ such that for $c\in(0,c_0)$,
if $f\in L^2(\Omega)$ then $f=u+d^s v$,
where $v\in\dot H^1_{\loc}(\Omega\setminus\{x=0\})$, while $u\in
L_\loc^2(\Omega\setminus\{x=0\})$ can be stably determined from
$If$ restricted to $\Omega$-local geodesics in the following sense. 
There is a
continuous map
$If\mapsto u$,  where for $s\geq 0$, $f$ in $H^s(\Omega)$, the $H^{s-1}$ norm of $u$
restricted to any compact subset of $\Omega\setminus\{x=0\}$
is controlled by the $H^s$ norm of $If$ restricted to the set of $\Omega$-local geodesics.

Replacing $\Omega_c=\{\tilde x>-c\}\cap M$ by
$\Omega_{\tau,c}=\{\tau>\tilde x>-c+\tau\}\cap M$, $c$ can be taken
uniform in $\tau$ for $\tau$ in a compact set on which the strict
concavity assumption on level sets of $\tilde x$ holds.
\end{thm}

The uniqueness part of the theorem generalizes  Helgason's type of support theorems for tensors fields for analytic metrics \cite{Krish1,KrSt,BQ1}. In those works however, analyticity plays a crucial role and the proof is a form of a microlocal analytic continuation. In contrast, no analyticity is assumed here. We also present an inversion formula.

\medskip

\noindent{\it Idea of the proof:}

We introduce a Witten-type (in the sense of the Witten Laplacian) solenoidal gauge on the scattering
cotangent bundle, $\Tsc^*X$
or its second symmetric power, $\Sym^2\Tsc^*X$. Fixing $\digamma>0$, our gauge is
$$
e^{2\digamma /x}\delta^s e^{-2\digamma/x}f^s=0,
$$
or {\em the $e^{-2\digamma/x}$-solenoidal gauge}. (Keep in mind here
that $\delta^s$ is the adjoint of $d^s$ relative to a scattering metric.)
We are actually working
with
$$
f_\digamma=e^{-\digamma/x} f
$$
throughout; in terms of this the
gauge is
$$
\delta^s_\digamma f^s_\digamma=0,\qquad \delta^s_\digamma=e^{\digamma/x}\delta^s e^{-\digamma/x}.
$$

Rephrasing Theorem \ref{thm:local-linear-intro-1} in terms of the solenoidal gauge we get

\begin{thm}\label{thm:local-linear-intro}%(See
 % Theorem~\ref{thm:local-linear} for the proof and the formula.)
There exists $\digamma_0>0$ such that for $\digamma\geq\digamma_0$ the
following holds.

For $\Omega=\Omega_c$, $c>0$ small, the geodesic X-ray transform on
{\em $e^{2\digamma/x}$-solenoidal}
one-forms and symmetric 2-tensors $f\in e^{\digamma/x} L_\scl^2(\Omega)$,
i.e.\ ones satisfying $\delta^s (e^{-2\digamma/x} f)=0$, is injective, with a stability
estimate and a reconstruction formula.

In addition, replacing $\Omega_c=\{\tilde x>-c\}\cap M$ by
$\Omega_{\tau,c}=\{\tau>\tilde x>-c+\tau\}\cap M$, $c$ can be taken
uniform in $\tau$ for $\tau$ in a compact set on which the strict
concavity assumption on level sets of $\tilde x$ holds.
\end{thm}

With $v$ a locally defined function on the space of geodesics,
for one-forms we consider the map $L$
\begin{equation}\label{eq:L-forms}
L v(z)=\int \chi(\lambda/x)v(\gamma_{x,y,\lambda,\omega})g_{\scl}(\lambda\,\pa_x+\omega\,\pa_y)\,d\lambda\,d\omega,
\end{equation}
while for 2-tensors
\begin{equation}\label{eq:L-tensors}
L v(z)=x^{2}\int \chi(\lambda/x)v(\gamma_{x,y,\lambda,\omega})g_{\scl}(\lambda\,\pa_x+\omega\,\pa_y)\otimes g_\scl(\lambda\,\pa_x+\omega\,\pa_y)\,d\lambda\,d\omega,
\end{equation}
so in the two cases $L$ maps into one-forms, resp.\ symmetric
2-cotensors,
where $g_{\scl}$ is a {\em scattering metric} used to
convert vectors into covectors. The proof of Theorem \ref{thm:local-linear-intro} relies on the next proposition on the ellipticity of some scattering pseudodifferential operator analogous to the operator $A_\digamma$ of Section \ref{local geodesic function}.

\begin{prop}\label{prop:elliptic}
First consider the case of one forms. Let $\digamma>0$.
Given $\tilde\Omega$, a neighborhood of $X\cap M=\{x\geq 0,\ \rho\geq 0\}$ in $X$,
for suitable choice of the cutoff $\chi\in\CI_c(\RR)$ and of $M\in\Psisc^{-3,0}(X)$, the operator
$$
A_\digamma=N_\digamma+d^s_\digamma
M\delta^s_\digamma,\qquad N_\digamma=e^{-\digamma/x}LIe^{\digamma/x},\qquad d^s_\digamma=e^{-\digamma/x}d^s e^{\digamma/x},
$$
is elliptic in $\Psisc^{-1,0}(X;\Tsc^*X,\Tsc^*X)$ in $\tilde\Omega$.

On the other hand, consider the case of symmetric 2-tensors. Then
there exists $\digamma_0>0$ such that for $\digamma>\digamma_0$ the
following holds.
Given $\tilde\Omega$, a neighborhood of $X\cap M=\{x\geq 0,\ \rho\geq 0\}$ in $X$,
for suitable choice of the cutoff $\chi\in\CI_c(\RR)$ and of $M\in\Psisc^{-3,0}(X;\Tsc^*X,\Tsc^*X)$, the operator
$$
A_\digamma=N_\digamma+d^s_\digamma
M\delta^s_\digamma,\qquad N_\digamma=e^{-\digamma/x}LIe^{\digamma/x},\qquad d^s_\digamma=e^{-\digamma/x}d^s e^{\digamma/x},
$$
is elliptic in $\Psisc^{-1,0}(X;\Sym^2\Tsc^*X,\Sym^2\Tsc^*X)$ in $\tilde\Omega$.
\end{prop}

%%%%%%%%%%%%%%%%%%%%%%%%%%%%%%%%%%%%%%%%%%%%%%%%%%%%%%%%%%%%%%%%%%%%%%%%%%%%%%%%%%%%%

%%%%%%%%%%%%%%%%%%%%%%%%%%%%%%%%%%%%%%%%%%%%%%%%%%%%%%%%%%%%%%%%%%%%%%%%%%%%%%%%%%%

\subsection{The non-linear result for conformal metrics}\label{nonlinear conformal}

We move to the non-linear problem. In this section, we consider the boundary rigidity problem in the class of  metrics conformal to a given one and with partial (local) data, that is, we know the boundary 
distance function $d_g$ for points on the boundary near a given point. In \cite{SUV_localrigidity} Stefanov, Uhlmann and Vasy show that one can recover uniquely and in a stable way a conformal factor near a strictly convex point where we have the information. In particular, this implies that we can determine locally the isotropic sound speed of a medium by measuring the travel times of waves joining points close to a convex point on the boundary.

We assume that $\bo$ is strictly convex at $p\in\bo$ w.r.t.\ $g$. Then the  boundary rigidity and the  lens rigidity problems with partial data are equivalent: knowing $d$ near $(p,p)$ is equivalent to knowing $L$ in some neighborhood of  $S_p\partial M$. The size of that neighborhood however depends on a priori bounds of the derivatives of the metrics with which we work.  This equivalence was first noted by Michel \cite{Mi}, since the tangential gradients of $d(x,y)$ on $\bo\times\bo$ give us the tangential projections of $-v$ and $w$, see also \cite[sec.~2]{S-Serdica}. Note that local knowledge of $\ell$ is not needed for the lens rigidity problem\footnote{If $L$ is given only, then the problem is called \textit{scattering rigidity} in some works}, and in fact, $\ell$ can be recovered locally from either $d$ or $L$.% see for example the proof of  Theorem~\ref{thm_stab_global}.

\subsubsection{Pseudolinearization}

The starting point is an identity in \cite{StU1}. We will repeat the proof.

Let $V$, $\tilde V$ be two vector fields on a  manifold  $M$ (which will be replaced later with $S^*M$). Denote by $X(s,X\zero)$ the solution of $\dot X=V(X)$, $X(0)=X\zero$, and we use the same notation for $\tilde V$ with the  
corresponding solution are denoted by $\tilde X$. 
Then we have the following simple statement.

\begin{lemma}\label{SU-identity}
For any $t>0$ and any  initial condition $X\zero$, if $\tilde X\!\left(\cdot,X\zero\right)$ and $X\!\left(\cdot,X\zero\right) $ exist on the interval $[0,t]$, then 
\begin{equation} \label{L1}
\begin{split}
& \tilde X\!\left(t,X\zero\right) -X\!\left(t,X\zero\right)\\ &\quad  = \int_0^t \frac{\partial \tilde X}{\partial X\zero}\!\left(t-s,X(s,X\zero)\right)\left(\tilde V-V \right)\!\left( X(s,X\zero)\right)\,ds.
\end{split}
\end{equation}
\end{lemma} 
\begin{proof}
Set 
$$
F(s) = \tilde X\!\left(t-s,X(s,X\zero)\right).
$$
Then
\begin{align*}
F'(s)=& -\tilde V\!\left(\tilde X(t-s,X(s,X\zero))\right)\\ & + \frac{\partial \tilde X}{\partial X\zero} \!\left(t-s,X(s,X\zero)\right)V\!\left(X(s,X\zero)\right).
\end{align*}
The proof of the lemma would be complete by the fundamental theorem of calculus
$$
F(t)-F(0)=\int_0^t F'(s)\, ds
$$
 if we show the following
\begin{equation} \label{110a}
 \tilde V\!\left(\tilde X(t-s,X(s,X\zero))\right)
= \frac{\partial \tilde X}{\partial X\zero}\!\left(t-s,X(s,X\zero)\right)
\tilde V\!\left(X(s,X\zero)\right).
\end{equation}
Indeed, (\ref{110a}) follows from
$$
0=\left.\frac{d}{d\tau}\right|_{\tau=0} X(T-\tau,X(\tau,Z))
= -V(X(T,Z)) +\frac{\partial X}{\partial X\zero}(T,Z)
V(Z), \quad \forall T,
$$
after setting $T=t-s$, $Z= X(s,X\zero)$.
\end{proof} 

Let $c$, $\tilde c$ be two speeds. Then the corresponding metrics are $g= c^{-2} dx^2$, and  $\tilde g = \tilde c^{-2} dx^2$. The corresponding Hamiltonians and Hamiltonian vector fields are  
$$
H = \frac12 c^2g_0^{ij}\xi_i\xi_j , \qquad 
V = \left(c^2g_0^{-1} \xi, -\frac12\partial_x \left(c^2|\xi|_{g_0}^2\right)\right),
$$
and the same ones related to $\tilde c$. We used the notation  $|\xi|_{g_0}^2:= g_0^{ij}\xi_i\xi_j $.

We denote points in the phase space $T^*M$, in a fixed coordinate system, by $z = (x,\xi)$. We denote the bicharacteristic with initial point $z$ by $Z(t,z) = (X(t,z),\Xi(t,z))$. We can naturally think of the scattering relation $L$ and the travel time $\ell$ as functions on the cotangent bundle instead of the tangent one. Then we get the following.

%Then we get the identity already used  in \cite{SU-MRL}
%\begin{equation} \label{1}
%\tilde Z(t,z) - Z(t,z)= 
%\int_0^t \frac{\partial\tilde Z}{\partial z} (t-s,Z 
%(s,z))\big(\tilde V - V\big)(Z (s,z))\,ds.
%\end{equation}

\begin{prop}\label{pr1} 
Assume  
\begin{equation} \label{5a}
L(x_0,\xi^0) = \tilde L(x_0,\xi^0), \quad \ell(x_0,\xi^0) =\tilde \ell(x_0,\xi^0)
\end{equation}
for some $z_0= (x_0,\xi^0)\in \partial_-S^*M$. Then 
\begin{equation}\label{pseudo-linear}
\int_0^{ \ell(z_0) }  \frac{\partial\tilde Z}{\partial z} ( \ell(z_0)-s,Z 
(s,z_0) )\big(V -\tilde V\big)(Z (s,z_0))\,ds =0.
\end{equation}
\end{prop}

Introduce the exit times $\tau(x,\xi)$ defined as the minimal (and the only) $t> 0$ so that $X(t,x,\xi)\in\bo$. They are well defined near $S_p\bo$, if $\bo$ is strictly convex at $p$. 
%We need to write $\frac{\partial \tilde Z}{\partial z}( \ell(z)-s,Z(s,z))$ as a function of $(x,\xi)=Z(s,z)$. We have
%$$
%\frac{\partial \tilde Z}{\partial z}(\ell(z)-s,Z(s,z)) = \frac{\partial \tilde Z}{\partial z}(\tau(x,\xi),(x,\xi)). 
%$$
We take the second $n$-dimensional component on \eqref{pseudo-linear} and use the fact that $c^2|\xi|_{g_0}^2=1$ on the bicharacteristics related to $c$ to get, with $f=c^2-\tilde c^2$,
\begin{equation} \label{3}
J_if(\gamma):= \int \left( A_i^j(X(t),\Xi(t))(\partial_{x^j}f)(X(t))+ B_i (X(t),\Xi(t))f(X(t)) \right)\d t=0
\end{equation}
for any bicharacteristic  $\gamma = (X(t),\Xi(t))$ (related to the speed $c$) in our set, where
\begin{equation} \label{4}
\begin{split}
A_i^j\left(x,\xi\right) = &-\frac12 \frac{\partial\tilde  \Xi_i}{\partial \xi_j}(\tau(x,\xi),(x,\xi)) c^{-2}(x),
\\
B_i\left(x,\xi\right) = &\frac{\partial \tilde\Xi_i}{\partial x^j}(\tau(x,\xi),(x,\xi))g_0^{ik}(x) \xi_k\\&  -\frac12 \frac{\partial\tilde  \Xi_i}{\partial \xi_j}(\tau(x,\xi),(x,\xi))  (\partial_{x^j}   g_0^{-1}(x))\xi\cdot\xi.
\end{split} 
\end{equation}

The arguments above lead to the following linear problem: 

\textbf{Problem.} 
Assume \r{3} holds with some $f$ supported in $M$, for all geodesics  close to the ones originating from $S^*_{x_0}\partial M$ (i.e. initial point $x_0$ and all unit initial co-directions tangent to $\bo$). Assume that $\bo$ is strictly convex at $x_0$ w.r.t.\ the speed $c$. %Assume \r{7}.  % (it is enough to hold for $\xi\in S^*_{x_0}\partial M$). 
 Is it true that $f=0$ near $x_0$? 

We show below that the answer is affirmative.

\subsubsection{Non-liner result for conformal metrics}

We continue by generalizing \eqref{3} to regard the functions
$\pa_{x_j}f$ and $f$ entering into it as independent unknowns, while
restricting the transform to the region of interest $\Omega=\Omega_\level$. So let $\tilde J_i$ be defined by
\begin{align*}
\tilde J_i &(u_0,u_1,\ldots,u_n)(\beta) \\&:= \int_{\gamma_\beta} \left( A_i^j(X(t),\Xi(t))u_j(X(t))+ B_i (X(t),\Xi(t))u_0(X(t)) \right)\d t,
\end{align*}
where $\gamma_\beta$ is the geodesic with lift to $S \Omega$ having
starting point $\beta\in S \Omega$. Let $\tilde J=(\tilde J_1,\ldots,\tilde J_n)$.
This is a vector valued version of the geodesic X-ray transform
considered in \cite{UV2}, and described above, sending functions on $\Omega$ with values in
$\Cx^{n+1}$ to functions with values in $\Cx^n$.

%For $\chi\in\CI_0(\RR)$, $\chi\geq 0$,
%$\chi(0)>0$, $u=(u_0,u_1,\cdots , u_n)$, one considers the map
%$$
%Pu(\foliation,\loccoord)=\int_\RR\int_{\sphere^{n-2}} \foliation^{-2}\chi(\lambda/\foliation)(\tilde J u)(\foliation,\loccoord,\lambda,\omega)\,\d\lambda\,\d\omega.
%$$

The following local invertibility result holds, the proof is in the spirit of section \ref{linear conformal}.

\begin{prop}\label{pr_3.2}
There is $\level_0>0$ such that for $0<\level<\level_0$, if
$f\in H^1(\Omega_\level)$ and $\tilde J (f,\pa_1 f,\ldots,\pa_n
f)=0$, then $f=0$.

In fact, for $\digamma>0$, $s\geq 1$, there exist $\level_0>0$,
$k$ and $\ep>0$ such that the following holds. For $\delta>0$
there is $C>0$ such that if
$0<\level<\level_0$, $\Gamma_\pm$ is $\ep$-close to
$\Gamma_\pm^0$ in $C^k$, $\tilde\foliation$ is $\ep$-close to
$\tilde\foliation_0$ in $C^k$, then
$$
\|f\|_{e^{(\digamma+\delta)/\foliation} H^s(\Omega_\level)}\leq
C\|\tilde J(f,\pa_1 f,\ldots,\pa_n f)\|_{e^{\digamma/\foliation} H^{s}(\cM_\level)}.
$$

Moreover, with $\Omega_{\level,\rho_0}=\Omegaext_\level\cap\{\rho\geq
\rho_0\}$, and $\cM_{\level,\rho_0}$ being defined analogously to
$\cM_\level$ with $\pa M=\{\rho=0\}$ being replaced by
$\{\rho=\rho_0\}$, we have: for $\digamma>0$ and $s\geq 1$ there exist $\level_0>0$, $\rho_0<0$,
$k$ and $\ep>0$ such that the following holds. For $\delta>0$
there is $C>0$ such that if
$0<\level<\level_0$, $\Gamma_\pm$ is $\ep$-close to
$\Gamma_\pm^0$ in $C^k$, $\tilde\foliation$ is $\ep$-close to
$\tilde\foliation_0$ in $C^k$, then $f\in
H^{s+1}(\Omega_{\level,\rho_0})$
implies that
$$
\|f\|_{e^{(\digamma+\delta)/\foliation} H^s(\Omega_{\level,\rho_0})}\leq
C\|\tilde J(f,\pa_1 f,\ldots,\pa_n f)\|_{e^{\digamma/\foliation} H^{s}(\cM_{\level,\rho_0})}.
$$
\end{prop}

Based on the pseudolinearization process, Proposition \ref{pr_3.2} implies the following results for the non-linear problem in a fixed conformal class.

\begin{thm}\label{thm_1}Let $n=\dim M\ge3$,  let $c>0$, $\tilde c>0$ be smooth  and let  $\bo$ be strictly convex with respect to both $g=c^{-2}g_0$ and $\tilde g = \tilde c^{-2}g_0$ near a fixed $p\in\bo$. Let $d(p_1,p_2)= \tilde d(p_1,p_2)$ for $p_1$, $p_2$ on $\bo$ near $p$. Then $c=\tilde c$ in $M$ near $p$.
\end{thm}

This is the only known result for the boundary rigidity problem with partial data except in the case that the metrics are assumed to be real-analytic \cite{LSU}. The latter follows from determination of the jet of the metric at a convex point from the distance function known near $p.$

We have an immediate corollary of our main result for the lens rigidity problem. To reduce this problem to Theorem~\ref{thm_1} directly, we need to assume first that $c=\tilde c$ on $\bo$ near $p$ to make the definition of $\partial_\pm SM$ independent of the choice of the speed but in fact, one can redefine the lens relation in a way to remove that assumption, see \cite{StU5}.

\begin{thm}\label{corollary_1}
Let $M$, $c$, $\tilde c$ be as in Theorem~\ref{thm_1} with $c=\tilde c$ on $\bo$ near $p$. Let $L=\tilde L$ near $S_p\bo$. Then $c=\tilde c$ in $M$ near $p$. 
\end{thm}

We also prove  H\"older conditional stability estimates related to the uniqueness theorems above.  In case of data on the whole boundary, such an estimate was proved in  \cite[section~7]{StU4} for simple manifolds and metrics not necessarily conformal to each other.   Below, the $C^k$ norm is defined in a fixed coordinate system. The next theorem is a local stability result, corresponding to the local uniqueness result in Theorem~\ref{thm_1}. 

\begin{thm}\label{thm_stability}
There exists $k>0$ and $0<\mu<1$ with the following property. For any  $0<c_0\in C^k(M)$, $p\in\bo$, and $A>0$, there exists $\eps_0>0$ and $C>0$ with the property that for any two positive $c$, $\tilde c$ with 
\begin{equation} \label{est3}
\|c-c_0\|_{C^2} +\|\tilde c-c_0\|_{C^2} \le \eps_0, \quad \text{and} \quad \|c\|_{C^k}+ \|\tilde c\|_{C^k}\le A, 
\end{equation}
%with some $A>0$, for any $p\in \bo$ and an 
and for any neighborhood $\Gamma$ of $p$ on $\bo$, % $V\subset\bo\times\bo$,  
we have the stability estimate
\begin{equation} \label{stab}
\|c-\tilde c\|_{C^2(U)}\le C\|d-\tilde d\|_{C(\Gamma\times\Gamma)}^\mu
\end{equation}
for some neighborhood $U$ of $p$ in $M$.  
\end{thm}

%%%%%%%%%%%%%%%%%%%%%%%%%%%%%%%%%%%%%%%%%%%%%%%%%%%%%%%%%%%%%%%%%%%%%%%%%%%%%%%%%%

\subsection{Global result under the foliation condition}\label{global foliation}

Above linear and non-linear partial data results have immediate applications in the global problems under some global geometric condition. We first give the definition of the geometric condition necessary for our global theorems.

\begin{Def}\label{def_1.1}
Let $(M,g)$ be a compact Riemannian manifold with boundary. We say that
$M$ satisfies the foliation condition by strictly convex hypersurfaces if 
 $M$ is equipped with a smooth function $\rho: {M}\to[0,\infty)$ which level sets $\Sigma_t=\rho^{-1}(t)$, $t<T$
 with some $T>0$ are strictly convex viewed from $\rho^{-1}((0,t))$ for $g$,   $d\rho$ is non-zero on these level sets,  and $\Sigma_0=\partial M$ and $M\setminus\cup_{t\in[0,T)}\Sigma_t$  has empty interior. 
\end{Def}

The global geometric condition that we are imposing is a natural analog of the condition
\begin{equation} \label{HWZ}
\frac{\partial}{\partial r}\frac{r}{c(r)}>0, %\quad \text{for $0<r=|x|\le R$}
\end{equation}
with $\frac{\partial}{\partial r} = \frac{x}{|x|}\cdot\partial_x$ the radial derivative, proposed by Herglotz \cite{Her} and
Wiechert and Zoeppritz \cite{WZ} for an isotropic radial sound speed $c(r)$.
In this case the geodesic spheres are strictly convex. 

In fact \cite[Sec. 6]{SUV_localrigidity} extends the Herglotz and Wiechert \& Zoeppritz results to not necessarily radial speeds $c(x)$ satisfying \eqref{HWZ}. Let $ B(0,R)$, $R>0$ be the ball in $\R^n$, $n\ge3$ centered at the origin with radius $R>0$. % The background metric $g_0$ in this section is the Euclidean one. 
%Let $0<c(r)$ be smooth in $[0,R]$ so that $c(|x|)$ is smooth as well (at the origin). 
Let $0<c(x)$ be smooth in $B(0,R)$.
%\begin{equation} \label{HWZ}
%\frac{\partial}{\partial r}\frac{r}{c(r)}>0, \quad \text{for $0<r=|x|\le R$},
%\end{equation}
%where $\frac{\partial}{\partial r} = \frac{x}{|x|}\cdot\partial_x$ is the radial derivative. 
%We do not assume that $c$ is radial, i.e., that it depends on $r=|x|$ only. We show below that \r{HWZ} is in fact a foliation condition. 

\begin{prop}\label{pr_HWZ}
The Herglotz and Wieckert \& Zoeppritz condition \r{HWZ} for $0<r=|x|\leq R$ is equivalent to the the condition that the Euclidean spheres $S_r=\{|x|=r\}$ are strictly convex in the metric $c^{-2} \d x^2$ for $0<r\le R$. 
\end{prop}

Other examples of non-simple metrics that satisfy the foliation condition are the tubular neighborhood of a
closed geodesic in negative curvature. These have trapped geodesics. Also the rotationally symmetric spaces
on the ball with convex spheres can be far from simple. It follows from Lemma 1 of \cite{Eber}
that a simply connected manifold which has one point such that any geodesic emanating from that point is free of focal point satisfies the foliation condition. Such foliation condition also holds on complete noncompact manifolds with positive curvature \cite{GW}. It would be
interesting to know whether this is also the case for simple manifolds. As it was mentioned earlier manifolds
satisfying the foliation condition are not necessarily simple.
%It is also satisfied
%for negatively curved manifolds.
%But this condition allows in principle for conjugate points of the metric. 
In \cite{SU4}, \cite{HU} one can find a microlocal study of the geodesic X-ray transform with fold caustics.
A similar condition of foliating by convex hypersurfaces was used in \cite{StU6} to 
satisfy the pseudoconvexity condition needed for Carleman estimates. See \cite{UV2, SUV_localrigidity, SUV} for more details about the foliation condition.

%%%%%%%%%%%%%%%%%%%%%%%%%%%%%%%%%%%%%%%%%%%%%%%%%%%%%%%%%%%%%%%%%%%%%%%%%%%

\subsubsection{Global linear result}

The global linearized problem, tensor tomography problem, has been extensively studied in the literature for both simple and non-simple manifolds
\cite{Mu, Dairbekov, PSU, PSU_1, PS, SSU, Vladimir_97, Sh2, StU2, StU4a, SU4, Bao-Zhang, PZ15}. See the book \cite{Sh} and  \cite{PSU3} for a recent survey.
%We consider domains with compact closure $\overline{X}$ equipped with a function $\rho:\overline{X}\to[0,\infty)$ whose level sets $\Sigma_t=\rho^{-1}(t)$, $t<T$, are strictly convex (viewed from $\rho^{-1}((t,\infty))$ (and $d\rho$ is non-zero on these level sets), with $\Sigma_0=\pa X$ and $X\setminus\cup_{t\in[0,T)}\Sigma_t=\rho^{-1}([T,\infty))$ either having $0$ measure or having empty interior. (Note in particular that $\rho$ is a boundary defining function.)
In this section, we consider the global invertibility of the geodesic ray transform on manifolds satisfying the foliation condition. Assume $n=$ dim $M\geq 3$.

\begin{thm}
For $X$ and $\rho$ as above, if $X\setminus\cup_{t\in[0,T)}\Sigma_t$ has $0$ measure, the global geodesic transform is injective on $L^2(X)$, while if it has empty interior, the global geodesic transform is injective on $H^s(X)$, $s>n/2$.
\end{thm}

\begin{proof}
This global result is an immediate consequence of Theorem \ref{local geodesic function}. Indeed, if $If=0$ and $f\in H^s$, $s>n/2$, $f\neq 0$, then $\supp f$ has non-empty interior since $f$ is continuous by the Sobolev embedding, while if $f\in L^2$, $f\neq 0$, then $\supp f$ has non-zero measure. On the other hand, let $\tau=\inf_{\supp f}\rho$; if $\tau\geq T$ we are done, for then $\supp f\subset X\setminus\cup_{t\in[0,T)}\Sigma_t$. Thus, suppose $\tau<T$, so $f\equiv 0$ on $\Sigma_t$ for $t<\tau$, but there exists $q\in\Sigma_\tau\cap\supp f$ (since $\supp f$ is closed and $\overline{X}$ is compact). Now we use Theorem \ref{local geodesic function} on $\rho^{-1}(\tau,\infty)$ to conclude that a neighborhood of $q$ is disjoint from $\supp f$ to obtain a contradiction.
\end{proof}

In fact, in this global setting we can even take $\tilde x=-\rho$, and the uniformity of the constants in terms of $c$ and $\rho_0$, as stated in the remark after Theorem \ref{local geodesic function} directly yields that if $t<T$ then there exists $\delta=\delta_t>0$ such that if $c,\rho_0\in (t-\delta_t,t+\delta_t)$ then a stability estimate holds (with a reconstruction method!) for the region $\rho^{-1}([\rho_0,c))$. Now in general, for $T'<T$, one can take a finite open cover of $[0,T']$ by such intervals $(t_j',t_j'')$, $j=1,\ldots,k$ (with, possibly after some reindexing and dropping some intervals, $t_1'<0$, $t_k''>T'$, $t_j''\in(t_{j+1}',t_{j+1}'')$), and proceed inductively to recover $f$ on $\cup_{t\in[0,T']}\Sigma_t$ from its X-ray transform, starting with the outermost region. More precisely, first, using the theorem, one can recover the restriction of $f$ to $\rho^{-1}((-\infty,t_1''))$. Then one turns to the next interval, $(t_2',t_2'')$, and notes there is a reconstruction method for the restriction to $\rho^{-1}((t_2',t_2''))$ of functions $f_2$ supported in $\rho^{-1}((t_2',+\infty))$ (no support condition needed at the other end, $t_2''$). One applies this to $f_2=\phi_2 f$, where $\phi_2$ identically $1$ near $\rho^{-1}([t''_1,+\infty))$, supported in $\rho^{-1}((t_2',+\infty))$; since $f=(1-\phi_2)f+\phi_2 f$, and one has already recovered $(1-\phi_2)f$, one also knows the X-ray transform of $\phi_2 f$, and thus Theorem \ref{local geodesic function} is applicable. One then proceeds inductively, covering $\rho^{-1}([0,T'])$ in $k$ steps. This gives a {\em global} stability estimate, and indeed a reconstruction method doing a reconstruction layer by layer; that is, we have (in principle) developed a layer stripping algorithm for this problem.

We also remark that our approach is a completely new one to uniqueness for the global problem for the geodesic ray transform. The only method up to now, except in the real-analytic category \cite{StU_analytic}, has been the use of energy type equalities one introduced by Mukhometov \cite{Mu} and developed by several authors which are now called ``Pestov identities". 

\medskip

Similar global result holds for the geodesic ray transform of symmetric tensor fields. To state this, assume
that $\tilde x$ is a globally defined function with level sets $\Sigma_t$
which are strictly concave from the super-level set for $t\in (-T,0]$,
with $\tilde x\leq 0$ on the manifold with boundary $M$.
Then we have the following immediate consequence of Theorem \ref{thm:local-linear-intro-1} :

\begin{thm}\label{thm:global}
Suppose $M$ is compact.
The geodesic X-ray transform is injective and stable modulo potentials
on the
restriction of
one-forms and symmetric 2-tensors $f$ to $\tilde x^{-1}((-T,0])$ in
the following sense. For all $\tau>-T$ there is $v\in \dot H^1_\loc(\tilde
x^{-1}((\tau,0]))$ such that $f-d^sv\in L^2_\loc(\tilde
x^{-1}((\tau,0]))$ can be
stably recovered from $If$. Here for stability we assume that $s\geq
0$, $f$ is in an
$H^s$-space, the norm on $If$ is an
$H^s$-norm, while the norm for $v$ is an $H^{s-1}$-norm.
\end{thm}

\begin{rem}
This theorem, combined with Theorem~2 in \cite{StU5} (with a minor change --- the no-conjugate condition there is only needed to guarantee a stability estimate, and we have it in our situation), implies a local, in terms of a perturbation of the metric, lens rigidity uniqueness result near metric satisfying the foliation condition. 
\end{rem} 

%\begin{proof}
%For the sake of contradiction,
%suppose there is no $v$ as stated on $\tilde x^{-1}((\tau_0,0])$ for
%some $0>\tau_0>-T$, $If=0$, and let
%$$
%\tau=\inf\{t\leq 0:\ \exists v_t\in \dot H^1_{\loc}(\{\tilde x>t\})\ \text{s.t.}\
%f=d^sv_t\ \text{on}\ \{\tilde x>t\}\}\geq\tau_0.
%$$
%Thus, for any $\tau'>\tau$, such as
%$\tau'<\tau+c/3$, $c$ as in the uniform part of
%Corollaries~\ref{cor:local-linear-one-form}-\ref{cor:local-linear-2-tensor} on the levels $[\tau,0]$, there is $v\in \dot
%H^1_{\loc}(\{\tilde x>\tau'\})$ such that $f=d^s v$ on $\{\tilde
%x>\tau'\}$. Choosing $\phi\in\CI(M)$ identically $1$ near $\tilde
%x\geq\tau+2c/3$, supported in $\tilde x>\tau+c/3$, $f-d^s(\phi v)$ is
%supported in $\tilde x\leq\tau+2c/3$. But then by the uniform
%statement of Corollaries~\ref{cor:local-linear-one-form}-\ref{cor:local-linear-2-tensor}, there exists $v'\in
%\dot H^1_\loc (\{\tau-c/3<\tilde x\leq \tau+2c/3\})$ such
%that $f-d^s(\phi v)=d^sv'$ in $\tau-c/3<\tilde x<\tau+2c/3$. Extending
%$v'$ as $0$, the resulting function $\tilde v'\in \dot H^1_\loc
%(\{\tau-c/3<\tilde x\})$ and $d^s \tilde v'$ is the extension of $d^s
%v'$ by $0$. Thus, $f=d^s(\phi
%v+\tilde v')$, and this contradicts the choice of $\tau$, completing
%the proof.

%The stability of the recovery follows from a similar argument: by the
%uniform property one can recover $f$ modulo potentials in a finite
%number of steps: if $c$ works uniformly on $[\tau,0]$, at most
%$|\tau|/c+1$ steps are necessary.
%\end{proof}

%%%%%%%%%%%%%%%%%%%%%%%%%%%%%%%%%%%%%%%%%%%%%%%%%%%%%%%%%%%%%%%%%%%%%%%%%%

\subsubsection{Global non-linear result}

Now we use the layer stripping type argument to obtain a global result on lens rigidity problem which is different from Mukhometov's for simple manifolds.

%The statement of the global result on lens rigidity is as follows:

\begin{thm}\label{global nonlinear} 
Let $n=\dim M\ge3$,  let $c>0$, $\tilde c>0$ be smooth  and equal on $\bo$,  let $\bo$  be strictly convex with respect to both $g=c^{-2}g_0$ and $\tilde g = \tilde c^{-2}g_0$. Assume that   $M$ can be foliated by strictly convex hypersurfaces for $g$. Then if  $L= \tilde L$ on $ \partial_-SM$, we have  $c=\tilde c$ in $M$.
\end{thm}

\begin{proof}[Proof of Theorem \ref{global nonlinear}]
Theorem \ref{global nonlinear} is now an easy consequence of Theorem \ref{thm_1} using a layer stripping argument. Let $f=c^2-\tilde c^2$.
Assume $f\neq 0$, then $\supp f$ has non-empty
interior. On the other hand, let $\tau=\inf_{\supp f}\rho$; if $\tau= T$ we
are done, for then $\supp f\subset M\setminus\cup_{t\in[0,T)}\Sigma_t$. Thus, suppose $\tau<T$, so
$f\equiv 0$ on $\Sigma_t$ for $t<\tau$, but there exists $x\in\Sigma_\tau\cap\supp f$ (since $\supp f$ is closed).
We will show below how to use Theorem~\ref{corollary_1} on $M_\tau:= \rho^{-1}(\tau,\infty)$  to conclude that a neighborhood of $x$ is disjoint from $\supp f$ to obtain a contradiction. 

All we need to show is that the scattering relations $L_\tau$ and $\tilde L_\tau$ on $\Sigma_\tau$ coincide. Note that $\Sigma_\tau = \bo_\tau$ is strictly convex for $\tilde g$ as well because the second fundamental form for $\tilde g$ can be computed by taking derivatives from the exterior $\rho<\tau$, where $g=\tilde g$. Fix $(x_\tau,v_\tau)\in \partial_-SM_\tau$, see Figure~\ref{fig:local_lens_rigidity_pic2}.  The geodesic $\gamma_{x_\tau,v_\tau}(s)$  cannot hit $\Sigma_\tau$ again for negative ``times'' $s$ because otherwise, we would get a contradiction with the strict convexity at $\Sigma_t$, where $t$ corresponds to the smallest value of  $\rho$ on that geodesic between two contacts with $\Sigma_\tau$. 
Since $c=\tilde c$ outside $M_\tau$,  $\gamma_{x_\tau,v_\tau}(s)$  and $\tilde \gamma_{x_\tau,v_\tau}(s)$   coincide outside $M_\tau$ for  $s<0$. It is not difficult to see that this negative  geodesic ray must be non-trapping, i.e.,  $\gamma_{x_\tau,v_\tau}$ would hit $\bo$ for a finite negative time $s$  at some point and direction $(x,v)\in \partial_-SM$. In the same way, we show that the same holds for the positive part, $s>0$, of a geodesic issued from $L_\tau(x_\tau,y_\tau) =:(y_\tau,w_\tau)\in \partial_+SM_\tau$; and the corresponding point on $\partial_+SM$ will be denoted by $(y,w)$. Then, since $L(x,v)= (y,w)$,  we would also get $L_\tau(x_\tau,v_\tau) =  (y_\tau,w_\tau)= \tilde L_\tau(x_\tau,v_\tau)$. 
\begin{figure}[ht] % float placement: (h)ere, page (t)op, page (b)ottom, other (p)age
  \centering
  \includegraphics[scale=0.75]{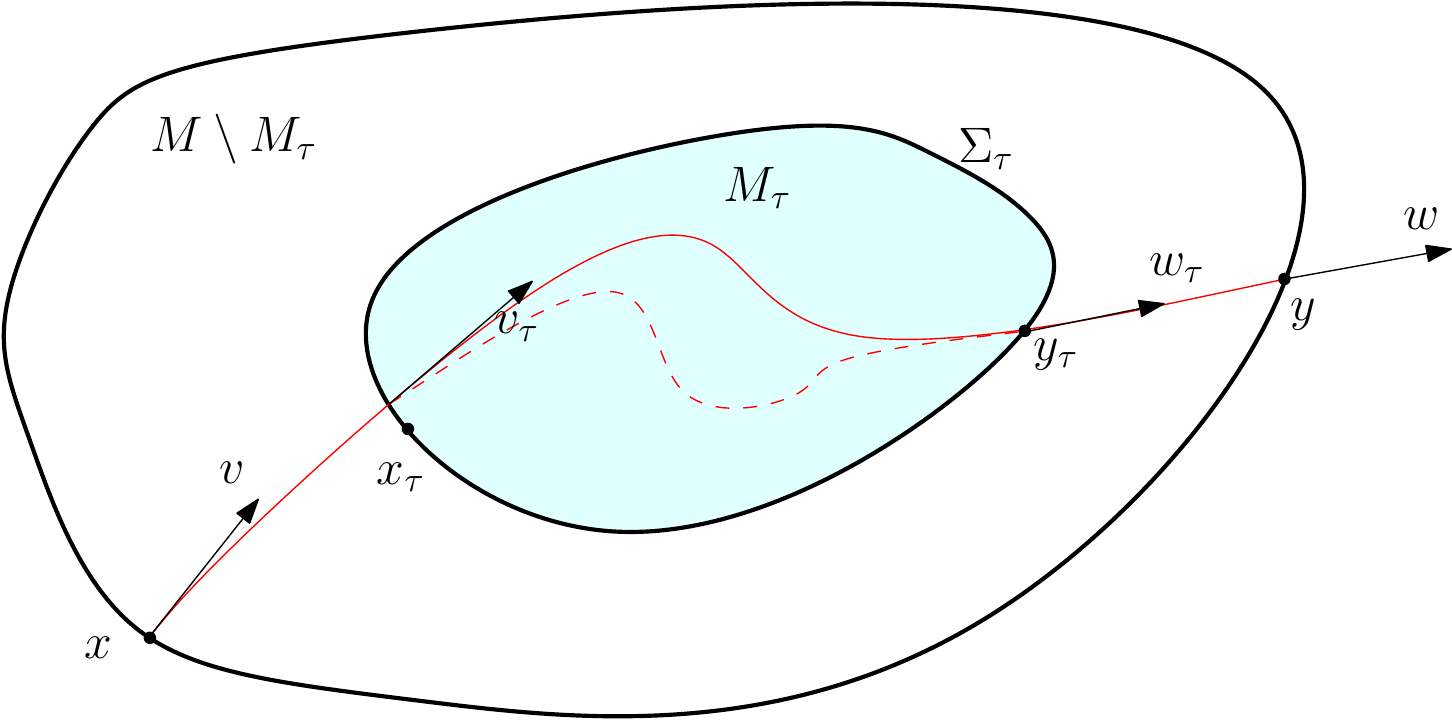}
\caption{One can recover the scattering relation on $\Sigma_\tau$  knowing that on $\bo$.}
  \label{fig:local_lens_rigidity_pic2}
\end{figure}
\end{proof}

A more general foliation condition under which the theorem would still hold is formulated in \cite{StU6}. In particular, $\Sigma_0$ does not need to be $\partial M$ and one can have several such foliations with the property that the  closure of their union is $M$. %If we can foliate only some connected neighborhood of $\bo$, we would get $c=\tilde c$ there.  
%Next,  it is enough to require that $M\setminus\cup_{t\in[0,T)}\Sigma_t$ is simple (or that it is included in a simple submanifold),   see the proof of Theorem~\ref{global nonlinear}   and Figure~\ref{fig:local_lens_rigidity_pic2}, to prove $c=\tilde c$ in $\cup_{t\in[0,T)}\Sigma_t$ first,  and then use Mukhometov's results to complete the proof. The class of manifolds we get in this way is larger than the simple ones, and can have conjugate points. 

%Speeds not necessarily radial (with $g_0$ the Euclidean metric) under the condition considered by Herglotz and Wieckert and Zoeppritz satisfy  the foliation condition of the theorem, see also Section~\ref{section_6}.   Other examples of non-simple metrics that satisfy the condition are the tubular neighborhood of a closed
%geodesic in negative curvature.  These have trapped geodesics.  
%%%%%%%%%%%%%%%%%%%%%%%%%%%%%%%%%%%%%%%%%%
%Also the rotationally symmetric spaces on the ball with convex spheres can be far from simple.
%%%%%%%%%%%%%%%%%%%%%%%%%%%%%%%%%%%%%%%%%%%%
%It follows from the  result of \cite{RS},  that manifolds with no focal points satisfy the foliation condition. It would be interesting to know whether this is also the case for simple manifolds. As it was mentioned earlier, manifolds satisfying the foliation condition are not necessarily simple. 

We have a H\"older conditional stability estimates  of global type as well, which can be considered as a ``stable version'' of Theorem~\ref{global nonlinear}. The $C^k$ norm below is defined in a fixed finite atlas of local coordinate charts.  In the same way we define $\dist(L,\tilde L)$ and its $C(D)$ norm: in any coordinate system we can just take the supremum of $L-\tilde L$ and then the maximum over all charts.  They can be defined in an invariant way, in principle but we do not do that for the sake of simplicity.

%For the purpose of the next theorem, we will extend and slightly generalize the foliation condition to compact submanifolds  $M_0$ of $M$. Let, as before, $\tilde M$ be a neighborhood of $M$, and extend $c$ smoothly there.  Note that the tilde over $M$ is not an indication that it is related to $\tilde c$. 

%\begin{Def}\label{def_5.1}
%Let $M_0\subset M$ be compact. We say that $M_0$ can be foliated by strictly convex hypersurfaces if there exists 
% a smooth function $\rho: \tilde M\to[0,\infty)$ which level sets $\Sigma_t=\rho^{-1}(t)$, $t\le T$
% with some $T>0$, restricted to $M_0$, are strictly convex viewed from $\rho^{-1}((0,t))$ for $g$;   $d\rho$ is non-zero on these level sets,   $\Sigma_0\cap M=\emptyset$, and $M_0\subset \rho^{-1}([0,T])$. 
%\end{Def}

%Note that this definition is not equivalent to Definition~\ref{def_1.1} when $M_0=M$ because in Definition~\ref{def_1.1}, we allow $M\setminus \rho^{-1}([0,T])$ to be non-empty (but with empty interior). Indeed, for uniqueness, proving $c=\tilde c$ outside such a set suffices since $c$ and $\tilde c$ are at least continuous. For stability however, it is convenient to assume that this set is empty.  

\begin{thm}\label{thm_stab_global}
Assume that $M_0\subset M$ can  be foliated by strictly convex hypersurfaces for $g=c^{-2}g_0$. Let $D\subset \partial_- SM$ be  a neighborhood of 
the compact set of all $\beta\in \partial_- SM\cap \partial_- SM_0$ consisting of the initial points of all geodesics $\gamma_\beta$ tangent to the intersections  of the strictly convex hypersurfaces with $M_0$. 
 Then with $k$, $\mu$, $c_0$, $c$, $\tilde c$, $\eps_0$ and $A$ as in Theorem~\ref{thm_stability}, we have the stability estimate
\begin{equation} \label{stab_global}
\|c-\tilde c\|_{C^2(M_0)}\le C\|\dist(L,\tilde L)\|_{C(D)}^\mu
\end{equation}
for $c$, $\tilde c$ satisfying \r{est3}. 
\end{thm}

%%%%%%%%%%%%%%%%%%%%%%%%%%%%%%%%%%%%%%%%%%%%%%%%%%%%%%%%%%%%%%%%%%%%55

%\subsection{Second linearization}

%\vspace*{3mm}

%\bibliographystyle{plain}

%\bibliography{sm}

%\end{document}

\end{document}